\documentclass[a4paper,two-side,8pt]{article}
\usepackage[latin1]{inputenc}
\usepackage[T1]{fontenc}
\usepackage{amsmath}
\usepackage{amsthm}
\usepackage{latexsym}
\usepackage{amssymb}
\usepackage[all]{xy}
\usepackage{calc}
\usepackage{multicol}
\newtheorem{thm}{Theorem}[section]
\newtheorem{prop}{Proposition}[section]

\newtheorem{lem}{Lemma}[section]

\newtheorem{rmq}{Remark}[section]

\newcommand{\R}{\mathbb{R}}
\numberwithin{equation}{section}

\newcounter{exercice}

\begin{document}
\title{Construction of a solution for the two-component radial Gross-Pitaevskii system with a large coupling  parameter}

\maketitle

\author{\hspace{1cm} \large Jean-Baptiste Casteras \footnote{D\'epartement de Math\'ematique, Universit\'e libre de Bruxelles, Campus de la Plaine CP 213, Bd. du Triomphe, 1050 Bruxelles, Belgium. E-mail: jeanbaptiste.casteras@gmail.com.}
\and  \hspace{2cm} Christos Sourdis \footnote{Department of Mathematics, University of Athens, Greece. E-mail: sourdis@uoc.gr} }


\begin{abstract}
We consider strongly coupled competitive elliptic systems that arise in the study of two-component Bose-Einstein condensates.
As the coupling parameter tends to infinity, solutions  that remain uniformly bounded are known
to converge  to a segregated limiting profile, with the difference of its components
satisfying
a limit scalar  PDE. In the case of radial symmetry, under natural non-degeneracy assumptions on a solution of the limit problem,
we establish by a perturbation argument  its persistence as a solution to the elliptic system.
\end{abstract}


\ \ \ \ \ {\bf AMSC}: 34B08, 34B15, 34E20

\section{Introduction}\subsection{The problem}
We consider coupled elliptic systems of the form
\begin{equation}\label{eqGen}
\Delta u_i=f_i(u_i)+gu_i\sum_{j\neq i}a_{ij}u_j^2, \ \textrm{in}\ \Omega;\ u_i=0\ \textrm{on}\ \partial \Omega,
\end{equation}
$i=1,\cdots,m$,
where $f_i$ are smooth functions with \begin{equation}\label{eqf0}f_i(0)=0 ,\end{equation} $g$ is a real parameter, $a_{ij}$ are nonnegative constants such that $a_{ii}>0$, $a_{ij}=a_{ji}$, $i,j=1,\cdots,m$, and $\Omega$ is a bounded smooth $N$-dimensional domain. Systems of this form arise in the study of multi-component Bose-Einstein condensates.
 In this context, the reaction terms are typically
\begin{equation}\label{eqreaction}
f_i(u)=g_iu^3-\mu_i u,\ g_i,\mu_i \in (-\infty,+\infty).
\end{equation} The coupling parameter $g$ measures the interaction between the different components in the mixture:
if $g<0$ they attract each other, whereas if $g>0$ they repel each other.
 On the other hand, the coefficients $g_i$ in (\ref{eqreaction}) measure the interaction between atoms in the same $i$-th component:
 if $g_i<0$ there is attraction, whereas if $g_i>0$ there is repulsion.

 The function $u_i$ represents the density corresponding to the $i$-th component in the mixture, and thus is naturally assumed to be positive. Nevertheless, the mathematical interest to (\ref{eqGen}) also extends to sign-changing solutions. In passing, we note that (\ref{eqGen}) has variational structure as it comes from a Gross-Pitaevskii energy.

In the following, we will consider the case of strong repulsion (or competition), that is $g\gg 1$. Moreover, we will focus on the case of two components, but first let us recall some of the main known results for the case of $m$ components.

\subsection{Known results}
In the seminal paper \cite{norisCPAM} (see also \cite{dancer1} for the corresponding parabolic problem), it was shown that if a family of solutions $\textbf{u}_g=(u_1^{g},\cdots,u_m^{g})$ of (\ref{eqGen}) with positive components remains bounded in $L^{\infty}(\Omega)$ as $g \to +\infty$, then it also remains bounded in $C^\alpha(\bar{\Omega})$ for any $\alpha \in (0,1)$. In fact, it was shown in \cite{SZARMA} that such families $\textbf{u}_g$ remain bounded, uniformly in $g$, even in the Lipschitz norm, at least away from the boundary of the domain.
 We also refer to \cite{WW1} for a related result in planar domains, and to \cite{CaffLin} for solutions that minimize the associated energy.
 Hence, thanks to a well known compact imbedding, possibly up to a subsequence $g_n\to +\infty $,  such a family converges in   $C^\alpha(\bar{\Omega})$ for any $\alpha <1$ to some limiting configuration $\textbf{u}_\infty=(u_1^{\infty},\cdots,u_m^{\infty})$. In fact, it was shown in \cite{norisCPAM} that the limiting profile has Lipschitz regularity up to the boundary of $\Omega$. Furthermore, the limiting components are segregated, that is their supports are disjoint. In its respective  support, the limiting component $u_i^\infty$  satisfies the following elliptic problem
\begin{equation}\label{eqf}
\Delta u_i^\infty=f_i(u_i^\infty).
\end{equation}
In the language of singular perturbations, the above limit problem  is called the  outer limit problem.


The regularity properties of the sharp interface \[\Gamma=\left\{x\in \bar{\Omega}\ : \ u_1^\infty(x)=\cdots=u_m^\infty(x)=0\right\}\]
were subsequently studied in \cite{tavaresCVPDE}.  It was shown there that $\Gamma$ has properties analogous to the nodal set of eigenfunctions of the Laplacian: there exists $\Sigma \subset \Gamma$ with $\mathcal{H}_{dim}(\Sigma)\leq N-2$ such that $\Gamma \setminus \Sigma$ is a finite union of smooth manifolds (we refer to \cite{zhang} for a detailed description of $\Sigma$). The set $\Sigma$ is referred to as the singular part of the interface $\Gamma$, whereas $\Gamma \setminus \Sigma$ as the regular part.
On each side of   a smooth manifold $M$ that composes the regular part of the interface there is only one nontrivial limiting component.
Moreover, across $M$ the corresponding limiting components, say $u_\infty=u_i^\infty$ and $v_\infty=u_j^\infty$ (it holds $i\neq j$, see \cite{dancer2}), satisfy the following reflection law:
\begin{equation}\label{eqreflect}
|\nabla u_\infty|=|\nabla v_\infty|\ \ \textrm{on}\ M. \end{equation}
We note that the above normal derivatives are nonzero by (\ref{eqf0}), (\ref{eqf}) and Hopf's boundary point lemma.



More refined estimates for the convergence as $g\to +\infty$ have recently been obtained in \cite{SZ} and \cite{wang}.
In particular, it was shown in the former reference that near a point $p$ of $M$, the two corresponding  components $u_g=u_i^g$, $v_g=u_j^g$ ($i \neq j$) that survive as $g\to +\infty$ should behave, to main order, in the following self-similar fashion:
\begin{equation}\label{eqinner}
u_g(x)\sim g^{-\frac{1}{4}}U\left(g^{\frac{1}{4}}\textrm{dist}(x,M)\right),\ \ v_g(x)\sim g^{-\frac{1}{4}}V\left(g^{\frac{1}{4}}\textrm{dist}(x,M)\right),
\end{equation}
 where $\textrm{dist}(\cdot,M)$ stands for the signed distance to $M$, while
the one-dimensional profiles $U(t),V(t)$  depend only on the point $p$
and satisfy
\begin{equation}\label{eqBUlimit}\left\{
\begin{array}{c}
  U''=UV^2 \\
  V''=VU^2
\end{array}\right.
\end{equation}
in the entire real line.
It was shown in \cite{berestycki1,berestycki2} that the above problem has just a 2-parameter family of positive solutions given by
\[
\mu U(\mu t+\tau),\ \mu V(\mu t+\tau),
\]
with scaling parameter $\mu>0$ and translation $\tau \in(-\infty,+\infty)$, for some fixed solution pair $(U,V)$ which satisfies
the mirror reflection symmetry \begin{equation}\label{eqRefUV}U(-t)\equiv V(t),\end{equation} and enjoys the following asymptotic behaviour at respective infinities:
\[
U(t)\to 0 \ \textrm{as}\ t\to -\infty;\ \ U'(t)\to |\nabla u_\infty(p)|>0 \ \textrm{as}\ t\to +\infty.
\]
Notice that the convergence in the previous limits is super-exponentially fast. In fact, it was observed in \cite{aftalionSourdis}
that there is an asymptotic phase $k=k(p)>0$ in the asymptotic behaviour of $U$ at $+\infty$. Combining all the previous information, we deduce that, for $t>0$ large enough,
\begin{equation}\label{missing}U(t)=|\nabla u_\infty (p)| t +k +O(e^{-c_1 t^2})\ \textrm{and}\ V(t)=O(e^{-c_2 t^2}), \end{equation}
for some positive constants $c_1$ and $c_2$. The above relations can be differentiated and, via (\ref{eqRefUV}), provide the corresponding asymptotic behaviour as $t\to -\infty$.



One also expects that the behaviour of solutions for large $g$ near $\Sigma$ should be governed by an equivariant entire solution with polynomial growth of the PDE  version of system (\ref{eqBUlimit}), see \cite{berestycki2,SZentire}, which is usually called the  inner (or blow-up) limit problem.

\subsection{The problem with two-components}
From now on, we will consider the special case of problem (\ref{eqGen}) with $m=2$, which (after a rescaling) we can write as
\begin{equation}\label{eq2comp}
\left\{
\begin{array}{ll}
  -\Delta u + f(u) + guv^2 = 0 & \\
  &\ \textrm{in}\ \Omega; \\
  -\Delta v + h(v) + gvu^2 = 0&
\\
&\\
 u=v=0&\ \textrm{on}\ \partial \Omega,
\end{array}\right.
\end{equation}
for some smooth functions $f$ and $h$  such that $f(0)=h(0)=0$ and $\Omega$ still a bounded, smooth $N$-dimensional domain.

We note that the reflection law (\ref{eqreflect}) implies that the difference
\[w=u_\infty-v_\infty \]
 is smooth across the regular part of the interface.
In fact, it was shown in \cite{dancer2} that this difference is a classical solution of the following limit problem
\begin{equation}\label{eqlimit}
{\Delta w=f(w^+)-h(-w^-)\ \textrm{in}\ \Omega;\ w=0\ \textrm{on}\ \partial \Omega},
\end{equation}
where one writes
\[w=w^++w^-\ \textrm{with}\ w^+\geq 0\ \textrm{and}\ w^-\leq 0.
\]
It is worthwhile mentioning that in the special case where $f\equiv h$ is odd, the above limit problem reduces to
\begin{equation}\label{eqlimRed}
{\Delta w=f(w) \ \textrm{in}\ \Omega;\ w=0\ \textrm{on}\ \partial \Omega}.
\end{equation}

\subsection{The converse problem}
So far we have discussed how one can reach the limit problem (\ref{eqlimit}) (and also (\ref{eqBUlimit})) starting from an appropriate family of solutions to (\ref{eq2comp}) for large $g$. It is also of interest whether one can go in the opposite direction, that is under which conditions do solutions of the limit problem (\ref{eqlimit}) generate corresponding solutions of (\ref{eq2comp}) for large values of $g$.

In \cite{dancer3}, Dancer considered (\ref{eq2comp}) for nonlinearities as in (\ref{eqreaction}) with $g_1,g_2>0$ (with the obvious correspondence with (\ref{eqGen})). It was shown by variational methods that, under appropriate restrictions on $\mu_1,\mu_2$, a certain type of nodal least energy solutions of (\ref{eqlimit}) generate corresponding  solutions with positive components to (\ref{eq2comp}) for large $g$. On the other hand, the authors of \cite{WW2} considered the case
where $g_1=g_2<0$ (say $-1$) and $\mu_1=\mu_2>0$ (say $1$) in a ball in two or three dimensions. In this case, it is well known that, for any integer $m\geq 1$, the (reduced) limit problem
(\ref{eqlimRed}) admits a radial  nodal solution $w_m$ with exactly an $m$ number of sign changes. Using variational methods, they were able to show that each $w_m$ produces a corresponding radial solution of (\ref{eq2comp}) with positive components that shadow respectively $(w_m)^+$ and $-(w_m)^-$ as $g\to +\infty$. A related result in $\mathbb{R}^N$, $N\geq 1$, can be found in \cite{mandel}.
%

At this point let us make a small detour and discuss briefly the analogous elliptic system modeling two competing populations that arises in spatial ecology. In that context,  the coupling terms in both equations of (\ref{eq2comp}) are $guv$, while the nonlinearities $f,h$ are usually of logistic type. Remarkably, uniformly bounded families of solutions to both systems share essentially the same regularity properties (with respect to large $g$), see \cite{conti}.
In particular, they have the same (outer) limit problem (\ref{eqlimit}). For the population problem, it was shown in \cite{dancerDu} by means of a topological  degree theoretic argument that non-degenerate (in the sense that the linearized operator does not have a kernel) nodal solutions $w$ of (\ref{eqlimit}) give corresponding solutions $(u_g,v_g)$ with positive components for the system with large $g$. The key idea for proving this is to consider the difference $u-v$ and note that this leads to a system with only one singularly perturbed equation  (a standard slow-fast system in the language of dynamical systems). Interestingly enough, this result was established without making use of the analogous blow-up limit problem to (\ref{eqBUlimit}). In light of the aforementioned common features of the two systems, it is natural to expect that an analogous converse result should also hold for the condensate problem (\ref{eq2comp}), see \cite{dancer4}.

\subsection{Main result}
We  show that an analogous converse result holds  for the condensate problem (\ref{eq2comp}), provided that we restrict to the radial setting and we impose some extra but milder  non-degeneracy assumptions on the solution of the limit problem (\ref{eqlimit}).

Our main result is the following.

\begin{thm}\label{thm}
Let $\Omega$ be an $N$-dimensional ball or annulus, $N\geq 1$, and let $f,h \in C^{3,1}[0,\infty)$ be such that $f(0)=h(0)=0$.  Suppose that $w$ is a radial  nodal solution of the limit problem
(\ref{eqlimit})
with  one sign change, which is non-degenerate in the radial class in the sense that the associated linearization does not have a nontrivial radially symmetric element in its kernel.
 Moreover, assume that $-w^-$ and $w^+$ are also non-degenerate in the radial class as solutions of (\ref{eqlimit}) in their respective supports.
Then, if $g$ is sufficiently large, there exists a radial solution $(u_g,v_g)$ of (\ref{eq2comp}) with positive components such that
\[ \|v_g+w^- \|_{L^\infty(\Omega)}\leq Cg^{-\frac{1}{4}},\ \ \|u_g-w^+ \|_{L^\infty(\Omega)}\leq Cg^{-\frac{1}{4}},  \]
where the constant $C>0$ is independent of $g$.

If $r_0$ denotes the radius of the sphere where $w$ vanishes, and $(r-r_0)w(r)>0$ for $r\neq r_0$, it holds
\[
\left\{
\begin{array}{c}
u_g(r)=g^{-\frac{1}{4}}U\left(g^{\frac{1}{4}}(r-r_0)\right)+O\left(g^{-\frac{1}{2}}+(r-r_0)^2\right)
\\
v_g(r)=g^{-\frac{1}{4}}V\left(g^{\frac{1}{4}}(r-r_0)\right)+O\left(g^{-\frac{1}{2}}+(r-r_0)^2\right)
\end{array}
\right.
\]
for $|r-r_0|\leq (\ln g)g^{-\frac{1}{4}}$, as $g\to +\infty$, where the pair $(U,V)$ is the unique solution of (\ref{eqBUlimit})
satisfying (\ref{eqRefUV}) and (\ref{missing}) with $u_\infty=w^+$ and $|p|=r_0$.
\end{thm}

Our result was announced in \cite{casterasEquad}.

We will only prove our result in the case where $\Omega$ is the unit ball $B_1$ centered at the origin, the case of an annulus is completely analogous.

As we will describe in more detail in the sequel, our proof relies on a perturbative method. We first combine the outer and inner problems, (\ref{eqlimit}) and (\ref{eqBUlimit}) respectively, to construct a sufficiently good approximate solution to (\ref{eq2comp}) for large $g$ that is valid in the whole domain. Then, we can capture a genuine solution nearby by a fixed point argument owing to appropriate invertibility properties of the associated linearized operator between carefully chosen weighted spaces.

We point out that the separate non-degeneracy assumptions on $-w^-$ and $w^+$ were not present in the previously mentioned result of
\cite{dancerDu} for the population system. As  will become apparent shortly, the underline reason for imposing them is the presence of the positive asymptotic phase $k$ in the asymptotic behaviour of the blow-up profile (recall (\ref{missing})). We point out that there was no such phase present in the analogous blow-up limits for the population problem. Loosely speaking, the outer and inner approximate solutions, given to main order by $(  w^{+},-w^{-} )$ and the pair in the righthand side of  (\ref{eqinner}) with $M=\{|x|=r_0\}$, respectively,
 do not have the phase $k>0$ in common (in the intermediate zone where they must match). Therefore, we need to move the outer solutions towards the inner one by a regular perturbation to compensate for the gap caused by $k>0$ (in principle, the inner solution should control the outer ones). To be able to do so, we need these non-degeneracy assumptions on $-w^-$ and $w^+$.
We remark that the non-degeneracy assumptions for $\pm w^{\pm}$ are much easier to verify in practice (see for instance \cite{Uniq1}) in comparison to that for $w$ which is a sign-changing solution (see \cite{uniq2Tanaka}); see also Section \ref{secapplications} below.

We believe that an analogous result still holds when $w$ changes sign an arbitrary number of times, provided one imposes further analogous non-degeneracy assumptions
to take into account the interaction created by adjacent zeros of $w(r)$ for $1 \ll g<\infty$.

\subsection{Method of proof} As mentioned above, our proof is perturbative and based on the construction of a sufficiently good approximate solution to (\ref{eq2comp}) for large $g$. We will first construct outer approximate solutions of the form
\[
(0,v_{\tilde{\delta}})\approx (0,v_0):=(0,-w^-)\ \ \textrm{and}\ \ (u_\delta,0)\approx (u_0,0):=(w^+,0)
\]
in the ball $B_{r_0}$ and the annulus $B_1\setminus \overline{B_{r_0}}$, respectively, depending on two free small parameters $\tilde{\delta},\delta>0$.
We point out that the aforementioned free parameters come from being able to regularly perturb the solutions $v_0$ and $u_0$ of (\ref{eqlimit}) in their respective supports, which is possible by the separate nondegeneracy assumptions in Theorem \ref{thm}  for $-w^-$ and $w^+$.
Actually, by their construction, $(0,v_{\tilde{\delta}})$ and $(u_\delta,0)$ will satisfy (\ref{eq2comp}) exactly in their respective domains of definition.
Since the singular limit $(-w^-,w^+)$ has a discontinuous gradient across the interface $\{|x|=r_0\}$, our outer approximate solution loses its effectiveness there.
To fix this issue, we will insert around the interface an inner approximate solution of the form \[\left(\mu g^{-\frac{1}{4}} U\left(\mu g^\frac{1}{4}(r-r_0-\xi) \right),\mu g^{-\frac{1}{4}} V\left(\mu g^\frac{1}{4}(r-r_0-\xi) \right)\right)+\textrm{higher\ order\ terms},\]
where $(U,V)$ solves (\ref{eqBUlimit}), with the scaling $\mu>0$ and the translation $\xi$ as free parameters. The higher order terms are determined by solving linear inhomogeneous problems involving the linearization of (\ref{eqBUlimit}) about $(U,V)$. An invertibility theory for such problems was developed in \cite{aftalionSourdis} based on the fact that the only bounded elements of the kernel of this operator are constant multiples of $(U',V')$ (see \cite{berestycki1}). The latter property means that the solution $(U,V)$ is linearly nondegenerate with respect to the ODE system (\ref{eqBUlimit}). In passing, we note that it is also nondegenerate with respect to the PDE version of (\ref{eqBUlimit}) (see \cite{SJDE}).
We want the outer and inner approximate solutions to match sufficiently well in an intermediate zone, so that we can glue them together via cutoff functions. In fact, this is where the main effort is needed.
To this end, we need to adjust conveniently the four free parameters that are involved in their construction. This task boils down to solving $4\times 4$ linear algebraic systems, which are solvable thanks to the first nondegeneracy assumption in Theorem \ref{thm} concerning $w$.

Having constructed a sufficiently good approximate solution to (\ref{eq2comp}) for large $g$, our next task is to linearize the problem about this approximation.
We find that the linearized operator is invertible and we obtain estimates for its inverse between appropriate weighted spaces. We stress that the previously mentioned linear nondegeneracy of the solution $(U,V)$ with respect to the ODE system (\ref{eqBUlimit}), as well as the nondegeneracy assumption on $w$ in Theorem \ref{thm}, both play a very important role in the analysis. On the other hand, the separate nondegeneracy assumptions for $-w^-$ and $w^+$ do not enter in the linear analysis.

Lastly, armed with the above, we can set up a contraction mapping argument and capture a true solution of (\ref{eq2comp}) near the approximate one.

\subsection{Outline of the paper} In the rest of the paper we will prove Theorem \ref{thm} in the case where $\Omega$ is the unit ball $B_1$ centered at the origin.
In Section \ref{secApprox} we will construct our approximate solution. In Section \ref{seclinear} we will study the linearized problem about it. In Section \ref{secnonlinear} we will apply a fixed point argument to the nonlinear problem and prove our main result. Finally, in Section \ref{secapplications} we will discuss some concrete situations where Theorem \ref{thm} is applicable.

\subsection{Notation} In the sequel we will use the following notation:
\[
v_0=-w^-\ \textrm{for}\ |x|<r_0,\ u_0=w^+\ \textrm{for}\ r_0<|x|<1.
\]
Since we will only be concerned with radial functions, we will denote $r=|x|$ and use $'$ to also symbolize differentiation with respect to $r$. Therefore, we have the following relation which will be used frequently in the sequel:
\[
-v_0'(r_0)=u_0'(r_0)=:\psi_0>0.
\]

By $C/c$ we will denote a large/small generic positive constant, independent of sufficiently large $g>0$,
whose value may increase/decrease as the paper progresses.

\section{Construction of the approximate solution $(u_{ap},v_{ap})$}\label{secApprox}
In this section, we are going to construct our approximate solution  $(u_{ap},v_{ap})$ to (\ref{eq2comp}) for large $g>0$.

\subsection{The outer approximate solution $(u_{out},v_{out})$} In $(0,r_0)$ and $(r_0,1)$, our outer approximate solution $(u_{out},v_{out})$ will be of the form $(0,v_{\tilde{\delta}})$ and $(u_\delta ,0)$, respectively, where these two functions are solutions to
\begin{equation}
\label{outapprox}
\begin{cases}\Delta v_{\tilde{\delta}}=h(v_{\tilde \delta}),\ r\in  (0,r_0),\\ \ v_{\tilde{\delta }}(r_0)=\tilde \delta , \end{cases}  \begin{cases} \Delta u_{\delta}=f (u_{ \delta}),\ r\in (r_0,1),\\ u_{ \delta} (r_0)=\delta,\ u_{\delta }(1)=0 ,  \end{cases}
\end{equation}
for some parameters $0\leq \delta ,\tilde{\delta} \ll 1$ to be determined later on. The existence of such solutions $v_{\tilde{\delta}}$, $u_\delta$ and their smooth dependence on   $\tilde{\delta}$, $\delta\geq 0$ follow from the implicit function theorem, due to the assumption that $u_0$ and $v_0$ are non-degenerate solutions of the respective limit problem in  \eqref{outapprox} for $\tilde{\delta}=\delta=0$. Notice that we can expand these two functions with respect to $\delta$ and $\tilde{\delta}$ namely
\begin{equation}\label{outerexp1}\begin{cases}u_\delta =u_0+ \delta u_1 +\delta^2 u_2 +\delta^3 u_3 +O(\delta^4),\\ v_{\tilde{\delta}}= v_0 +\tilde{\delta} v_1 + \tilde{\delta}^2 v_2 +\tilde{\delta}^3 v_3 +O(\tilde{\delta}^4),\end{cases}\end{equation}
where the $u_i,v_i$, $i\in \{1,2,3\}$ are solutions to
$$\begin{cases}\Delta v_1=h^{\prime}(v_0)v_1,\ r\in  (0,r_0),\\ \ v_{1}(r_0)=1 , \end{cases}  \begin{cases} \Delta u_{1}=f ^{\prime}(u_{ 0})u_1,\ r\in (r_0,1),\\ u_{1} (r_0)=1,\ u_{1 }(1)=0 , \end{cases} $$
$$\begin{cases}\Delta v_2=h^{\prime}(v_0)v_2+\frac{1}{2}h^{\prime \prime} (v_0) v_1^2,\ r\in (0,r_0),\\ \ v_{2}(r_0)=0 , \end{cases}  \begin{cases} \Delta u_{2}=f ^{\prime}(u_{ 0})u_2 +\frac{1}{2}f^{\prime \prime} (u_0) u_1^2 ,\ r\in (r_0,1),\\ u_{2} (r_0)=0,\ u_{2 }(1)=0 , \end{cases} $$
and
$$\begin{cases}\Delta v_3=h^{\prime}(v_0)v_3+h^{\prime \prime} (v_0) v_1v_2+\frac{1}{6}h^{(3)} (v_0) v_1^3,\ r\in  (0,r_0),\\ \ v_{3}(r_0)=0 , \end{cases} $$ $$ \begin{cases} \Delta u_{3}=f ^{\prime}(u_{ 0})u_3+f^{\prime \prime} (u_0) u_1u_2 +\frac{1}{6}f^{(3)} (u_0) u_1^3 ,\ r\in (r_0,1),\\ u_{3} (r_0)=0,\ u_{3 }(1)=0. \end{cases} $$Once more, these problems are solvable thanks to our nondegeneracy assumption on $u_0$ and $v_0$.
Expanding now in $r$, and setting \[s= r-r_0,\] we find
\begin{equation}\label{outerexp}\begin{cases}u_\delta (r)=\delta+ s (\psi_0 +\delta u_1^\prime (r_0) +\delta^2 u_2^\prime (r_0) )+ \dfrac{s^2}{2}(u_0^{\prime \prime} (r_0) +\delta  u_1^{\prime \prime} (r_0) )+\frac{s^3}{6}u_0^{\prime \prime \prime}(r_0)+ \sum_{j=0}^{3} {O(\delta^j s^{4-j})}, \\ v_{\tilde{\delta}}(r)= \tilde{\delta}+ s (-\psi_0   + \tilde{\delta} v_1^\prime (r_0) +\tilde{\delta}^2 v_2^\prime (r_0) )+\dfrac{s^2}{2} (v_0^{\prime \prime} (r_0) +\tilde\delta  v_1^{\prime \prime} (r_0) ) +\frac{s^3}{6}v_0^{\prime \prime \prime}(r_0)+ \sum_{j=0}^{3} {O(\tilde{\delta}^j s^{4-j})},\end{cases}\end{equation} for $s>0$ and $s<0$, respectively.


\subsection{The inner approximate solution $(u_{in},v_{in})$}
 For $r-r_0$ small enough, more precisely for $|r-r_0|\leq |\ln g|g^{-1/4}$, our ansatz for an inner approximate solution will be of the form
\begin{equation}
\label{inapprox}
u_{in}(r)=\mu g^{-1/4}U(t) +\varphi (t),\ v_{in}(r)=\mu g^{-1/4}V(t) +\tilde\varphi (t), \text{ where } t=\mu g^{1/4}(r-r_0-\xi ),
\end{equation}
where $(U,V)$ is a solution to \eqref{eqBUlimit} as in Proposition \ref{propoUV} below, the parameters $\mu >0$, $\xi \in \R$, and the corrections $\varphi ,\tilde \varphi$ are free  to be determined in the following. We will work under the assumption that
\[
\xi=O(g^{-\frac{1}{4}}),\ \mu=1+O(g^{-\frac{1}{4}}),\ \delta=O(g^{-\frac{1}{4}}),\ \tilde{\delta}=O(g^{-\frac{1}{4}}),
\]
which will be verified a-posteriori. Abusing notation slightly, we will also denote by $'$ differentiation with respect to $t$.

The following two propositions will play an important role in the construction of our inner approximate solution.
\begin{prop}\label{propoUV}
There exists a unique solution $(U,V)$ of
(\ref{eqBUlimit}) with positive components that satisfy the mirror symmetry condition
(\ref{eqRefUV}) and $U'(+\infty)=\psi_0$. In fact, the asymptotic behaviour
\begin{equation}\label{missing2}U(t)=\psi_0 t +k +O(e^{-c t^2})\ \textrm{and}\ V(t)=O(e^{-c t^2}), \end{equation}
holds for some $k>0$.
\end{prop}
 \begin{prop}\label{proSymA}
	Given $(H,\tilde{H})\in  \left[C(\mathbb{R})\right]^2$ satisfying the exponential decay estimate
	\[
|H(t)|+|\tilde{H}(t)|\leq Ce^{-c|t|},\ \ t\in \mathbb{R},
\]
 there exists a  solution $(\Phi,\tilde{\Phi})\in \left[C^2(\mathbb{R})\right]^2$ to \begin{equation}\label{eqL}L\left(\begin{array}{c}
                                                                                                                         \Phi \\
                                                                                                                           \\
                                                                                                                         \tilde{\Phi}
                                                                                                                       \end{array}\right):= \left(\begin{array}{c}
			-\Phi''+V^2\Phi+2UV\tilde{\Phi}  \\
			\\
			-\tilde{\Phi}''+U^2\tilde{\Phi}+2UV\Phi  \\
		\end{array}\right)=\left(\begin{array}{c}
                                                                                                                                                  H \\
                                                                                                                                                    \\
                                                                                                                                                  \tilde{H}
                                                                                                                                                \end{array}
  \right),
		\ \ t\in \mathbb{R},
	\end{equation}
 such that
	\[\begin{array}{lll}
	\Phi (t)=a_++bt+\mathcal{O}(e^{-c't}), & \tilde{\Phi}(t)=\mathcal{O}(e^{-c't}) & \textrm{as}\ \ t\to +\infty;
	\\
	& & \\
	\Phi(t)=\mathcal{O}(e^{c't}),& \tilde{\Phi}(t)=a_-+bt+\mathcal{O}(e^{c't})& \textrm{as}\ \ t\to -\infty,
	\end{array}
	\]
	for  any $c' \in (0,c)$, where
	\[
	b=-\frac{1}{2\psi_0}\int_{-\infty}^{\infty}\left(U'H+V'\tilde{H} \right)dt,
	\]
	and $a_+,a_-$ satisfy \[
a_++a_-=-\frac{1}{2\psi_0}\int_{-\infty}^{\infty}\left((tU'+U)H+(tV'+V)\tilde{H}  \right)dt.
\]
\end{prop}

Proposition \ref{propoUV} was proven in \cite{berestycki1}, while the fact that $k>0$ was observed in \cite{aftalionSourdis}. Proposition \ref{proSymA} was  proven in \cite{aftalionSourdis}.

 First let us plug \eqref{inapprox} into \eqref{eq2comp}. We get
\begin{align*}
&-\Delta u_{in}+f(u_{in})+g u_{in} v_{in}^2 \\
&=-g^{\frac{1}{2}}\mu^2 \varphi^{\prime \prime}+ 2g^{\frac{1}{2}}\mu^2 UV \tilde{\varphi} +g^{\frac{3}{4}}\mu U \tilde{\varphi}^2+g\varphi (g^{-\frac{1}{4}}\mu V+\tilde{\varphi} )^2 \\
&-\dfrac{n-1}{\frac{t}{\mu g^{\frac{1}{4}}}+r_0+\xi}\mu (\mu U^\prime + g^{\frac{1}{4}} \varphi^\prime)+ f(g^{-\frac{1}{4}}\mu U +\varphi ),
\end{align*}
and an analogous relation for $v_{in}$.

Observe that
$$\dfrac{n-1}{\frac{t}{\mu g^{\frac{1}{4}}}+r_0+\xi} =\dfrac{n-1}{r_0}  ( 1- \dfrac{\xi}{r_0}- \dfrac{t}{\mu g^{1/4}r_0 }) +O\left( (\dfrac{\xi}{r_0}+ \dfrac{t}{\mu g^{1/4}r_0 } )^2 \right) ,$$
and
$$f(g^{-1/4}\mu U+\varphi)=f^\prime (0)(g^{-1/4} \mu U +\varphi)+O\left( (g^{-1/4}\mu U+\varphi )^2  \right). $$
 To take care of the terms $-\dfrac{n-1}{r_0}\mu^2 U^\prime$ and $-\dfrac{n-1}{r_0}\mu^2 V^\prime$ in the remainder that is left by $(u_{in},v_{in})$, we define $(\varphi_0 ,\tilde{\varphi}_0 )$ to be a solution to
$$L\begin{pmatrix}\varphi_0  \\ \tilde{\varphi}_0\end{pmatrix}=\begin{pmatrix}\dfrac{n-1}{r_0} U^\prime g^{-\frac{1}{2}} \\ \dfrac{n-1}{r_0} V^\prime g^{-\frac{1}{2}} \end{pmatrix},$$ where $L$ is the linear operator that is defined by the lefthand side of (\ref{eqL}).
  Letting $Z,\tilde{Z}$ be smooth functions satisfying
$$ Z(t)=0,\ t\leq -1,\ Z(t)=-\dfrac{n-1}{2 r_0}\psi_0 g^{-1/2}  t^2 ,\ t\geq 1,   \ \textrm{and}\ \tilde{Z}(t)\equiv -Z(-t),$$
we see that the pair $(\phi_0 , \tilde{\phi}_0 )$ defined as
$$  g^{-1/2}(\phi_0 , \tilde{\phi}_0 )= (\varphi_0 , \tilde{\varphi_0}) -(Z, \tilde{Z})  , $$
satisfies
$$L\begin{pmatrix}\phi_0  \\ \tilde{\phi}_0 \end{pmatrix}=\begin{pmatrix}F_0\\ \tilde{F}_0 \end{pmatrix},$$
for some fixed functions $F_0,\ \tilde{F}_0$ such that
$$|F_0 (t)|+|\tilde{F}_0 (t)|\leq C e^{-ct^2},\  \tilde{F}_0(-t)=-F_0(t), \ t\in \R.$$
By Proposition \ref{proSymA}, we know that there exists a two parameter family of solutions with  $\tilde{\phi}_0(-t)\equiv-\phi_0(t)$ to the previous problem, parametrized by $A_0,B_0 \in \R$, such that
$$\begin{cases}\phi_0 (t)= b_0 t +A_0 \psi_0 +2B_0 \psi_0 t +B_0 k +O\left( (|A_0|+|B_0 |+1)e^{-c^\prime |t| } \right),\\ \tilde{\phi}_0 (t)=O\left( (|A_0 |+|B_0 |+1)e^{-c^\prime |t| } \right), \end{cases}$$
when $t\rightarrow +\infty$, and
$$\begin{cases}\tilde{\phi}_0 (t)=b_0 t -A_0\psi_0 -2B_0 \psi_0 t +B_0 k +O\left( (|A_0 |+|B_0 |+1)e^{-c^\prime |t| } \right),\\ \phi_0 (t)=O\left( (|A_0 |+|B_0 |+1)e^{-c^\prime |t| } \right), \end{cases}$$
when $t\rightarrow -\infty$, for some $b_0\in \R$ and any $c^\prime >0$. Combining all the previous computations and recalling that $s= r-r_0  $, we get that, for $t$ large enough,
\begin{equation}
\label{innerexp1}
\begin{cases}
\varphi_0 (t) &= -\dfrac{n-1}{2 r_0}\psi_0 g^{-1/2}  t^2  + g^{-1/2}(b_0 t + A_0 \psi_0 +B_0 k +2B_0 \psi_0 t )+O\left(g^{-\frac{1}{2}} (|A_0 |+|B_0 |+1)e^{-c^\prime |t| } \right)\\
&=  -\dfrac{n-1}{2 r_0}\psi_0 \mu^2  \xi^2  - g^{-1/4}\mu \xi(b_0+2B_0 \psi_0  )+ g^{-1/2} ( A_0 \psi_0 +B_0 k)\\
&+ s( \dfrac{n-1}{ r_0}\psi_0 \mu^2 \xi + \mu g^{-1/4} (b_0 +2B_0 \psi_0 ) )\\
&+s^2 (-  \dfrac{n-1}{2 r_0}\psi_0 \mu^2   )+O\left(g^{-\frac{1}{2}} (|A_0 |+|B_0 |+1)e^{-c^\prime |t| } \right),\\
\tilde{\varphi}_0 (t)&=O\left(g^{-\frac{1}{2}} (|A_0 |+|B_0 |+1)e^{-c^\prime |t| } \right),
\end{cases}
\end{equation}
and, for $-t$ is large enough,
\begin{equation}\label{innerexp2}
\begin{cases}
\tilde{\varphi}_0 (t) &=  \dfrac{n-1}{2 r_0}\psi_0 \mu^2  \xi^2  - g^{-1/4}\mu \xi(b_0-2B_0 \psi_0  )+ g^{-1/2} ( -A_0 \psi_0 +B_0 k)\\
&+ s(- \dfrac{n-1}{ r_0}\psi_0 \mu^2 \xi + \mu g^{-1/4} (b_0 -2B_0 \psi_0 ) )\\
&+s^2 (  \dfrac{n-1}{2 r_0}\psi_0 \mu^2   )+O\left(g^{-\frac{1}{2}} (|A_0 |+|B_0 |+1)e^{-c^\prime |t| } \right),\\
\varphi_0 (t)&=O\left(g^{-\frac{1}{2}} (|A_0 |+|B_0 |+1)e^{-c^\prime |t| } \right).
\end{cases}
\end{equation}
As a second approximation of our inner solution, we take
$$\begin{cases}u_{in,0}(r)= \mu g^{-1/4}U(t)+\varphi_0 (t),\\ v_{in,0}(r)= \mu g^{-1/4}V(t)+\tilde{\varphi}_0 (t) .\end{cases}$$
Let us also observe that
\begin{equation}\label{innerexp3}\begin{cases}\mu g^{-1/4} U(t)=\mu g^{-1/4} k - \psi_0 \mu^2 \xi + s (\mu^2 \psi_0)+O(g^{-\frac{1}{4}})e^{-g^\frac{1}{2}s^2},\ \text{when $t$ is large enough}, \\
\mu g^{-1/4} V(t)=\mu g^{-1/4} k + \psi_0 \mu^2 \xi + s (- \mu^2 \psi_0)+O(g^{-\frac{1}{4}})e^{-g^\frac{1}{2}s^2}, \text{ when $-t$ is large enough}.\end{cases} \end{equation}

Next, we define our third inner approximate solution by
$$\begin{cases}u_{in,0}(t)= \mu g^{-1/4}U(t)+\varphi_0 (t)+\varphi_1 (t),\\ v_{in,0}(x)= \mu g^{-1/4}V(t)+\tilde{\varphi}_0 (t)+\tilde{\varphi}_1 (t) ,\end{cases}$$
where $(\varphi_1,\tilde{\varphi}_1)$ is a solution to
{\small $$L\begin{pmatrix}\varphi_1  \\ \tilde{\varphi}_1\end{pmatrix}=\begin{pmatrix}- \dfrac{n-1}{r_0^2}\xi  U^\prime g^{-\frac{1}{2}}- \frac{n-1}{r_0^2} \mu^{-1} \psi_0 g^{-3/4} t - (\frac{n-1}{r_0})^2 \mu^{-1}\psi_0 g^{-3/4}t+\frac{n-1}{r_0}\mu^{-1} g^{-3/4} (b_0+ 2B_0 \psi_0) -\mu^{-1}g^{-3/4} f^\prime (0) U, \\ - \dfrac{n-1}{r_0^2}\xi  V^\prime g^{-\frac{1}{2}}+ \frac{n-1}{r_0^2} \mu^{-1} \psi_0 g^{-3/4} t+ (\frac{n-1}{r_0})^2 \mu^{-1}\psi_0 g^{-3/4}t+\frac{n-1}{r_0}\mu^{-1} g^{-3/4} (b_0- 2B_0 \psi_0) -\mu^{-1}g^{-3/4} h^\prime (0) V. \end{pmatrix}.$$}
Proceeding as we did for $\varphi_0$, we obtain, for $t$ large enough,
\begin{align}
\label{innerexp4}
\varphi_1 &= \dfrac{n-1}{r_0^2}\xi  \psi_0 g^{-\frac{1}{2}}t^2 /2  + \frac{n-1}{r_0^2} \mu^{-1} \psi_0 g^{-3/4} t^3 /6 + (\frac{n-1}{r_0})^2 \mu^{-1}\psi_0 g^{-3/4}t^3 /6\nonumber\\
&-\frac{n-1}{r_0}\mu^{-1} g^{-3/4} (b_0+ 2B_0 \psi_0)t^2 /2 +\mu^{-1}g^{-3/4} f^\prime (0) (k t^2 /2 +\psi_0 t^3 /6)\nonumber \\
&+ g^{-3/4} (a_1+b_1 t +A_1\psi_0 +2 B_1 \psi_0 t +B_1 k)+O\left( (1+\sum_{i=0}^1 |A_i|+|B_i|) e^{-c^\prime t }\right)\nonumber \\
&= \dfrac{n-1}{2 r_0^2}\xi^3  \psi_0 \mu^2   -( \frac{n-1}{ r_0^2}+(\frac{n-1}{r_0})^2 ) \mu^{2} \psi_0 \xi^3 /6\nonumber \\
&-\frac{n-1}{2r_0}\mu \xi^2  g^{-1/4} (b_0+ 2B_0 \psi_0)+\mu g^{-1/4} f^\prime (0) (k \xi^2 /2 -g^{1/4}\mu \psi_0 \xi^3 /6)\nonumber \\
&-g^{-1/2}\mu ( b_1 \xi   + 2 B_1 \psi_0 \xi ) + g^{-3/4}(a_1+A_1\psi_0 +B_1 k)\nonumber \\
&+ s ( - \dfrac{n-1}{ r_0^2}\xi^2  \psi_0 \mu^2   +( \frac{n-1}{ r_0^2}+(\frac{n-1}{r_0})^2 ) \mu^{2} \psi_0 \xi^2 /2 \nonumber \\
&\ \ \ \ +\frac{n-1}{r_0}\mu \xi  g^{-1/4} (b_0+ 2B_0 \psi_0)+\mu g^{-1/4} f^\prime (0) (-k \xi + g^{1/4}\mu \psi_0 \xi^2 /2)\nonumber \\
&\ \ \ \  +g^{-1/2}\mu ( b_1    + 2 B_1 \psi_0  )  ) \nonumber \\
&+ s^2 (        - (\frac{n-1}{r_0})^2  \mu^{2} \psi_0 \xi /2 \nonumber\\
&\ \ \ \ -\frac{n-1}{r_0}\mu   g^{-1/4} (b_0/2 + B_0 \psi_0)+\mu g^{-1/4} f^\prime (0) (k/2  - g^{1/4}\mu \psi_0 \xi /2) )\nonumber \\
&+ s^3 ( (\frac{n-1}{r_0^2}+(\frac{n-1}{r_0})^2 )  \mu^{2} \psi_0 /6 + \mu^2 f^\prime (0) \psi_0 /6  ) \nonumber\\
&+O\left( (1+\sum_{i=0}^1 |A_i|+|B_i|) e^{-c^\prime t }\right),
\end{align}
for some fixed $a_1,b_1\in \R$, any $c^\prime$ and some parameters $A_1,B_1$. We obtain a similar expansion for $\tilde{\varphi}_1$, replacing $\psi_0$ by $-\psi_0$, when $-t$ is large enough.

It is easy to see that the following proposition holds.
\begin{prop}\label{propInnerRemain}
The remainder that is left by the inner approximate solution in the system satisfies
\[
\left(\begin{array}{l}
  -\Delta u_{in} + f(u_{in}) + gu_{in}v_{in}^2  \\ \\ -\Delta v_{in} + h(v_{in}) + gv_{in}u_{in}^2
\end{array}\right)=\left(\begin{array}{c}
                           O(|\ln g|^4g^{-\frac{1}{2}}) \\
                           \\
                            O(g^{-\frac{1}{2}})e^{-c'g^\frac{1}{4}(r-r_0)}
                         \end{array}
\right),
\]
uniformly for $0\leq r-r_0\leq C|\ln g|g^{-\frac{1}{4}}$, for any $c'>0$, provided that $g$ is sufficiently large. An analogous estimate holds for $-C|\ln g|g^{-\frac{1}{4}}\leq r-r_0\leq 0$.
\end{prop}

\subsection{Matching the outer and inner approximate solutions: Adjusting the free parameters} Here we will determine all the previous free parameters, namely
$\delta_i, \tilde{\delta}_i,A_j,B_j $, $i=1,2,3$, $j=0,1$, in order to match $(u_{out},v_{out})$ with $(u_{in}, v_{in})$ when $|r-r_0|\approxeq |\ln g|g^{-\frac{1}{4}}$.

In light of the asymptotics  \eqref{innerexp1}, \eqref{innerexp2}, \eqref{innerexp3} for $(u_{in},v_{in})$ and those in  \eqref{outerexp} for $(u_{out},v_{out})$, by equating the powers $s^0$ and $s^1$ in these expansions,  setting \[\delta =\delta_1 +\delta_2+\delta_3,\ \tilde \delta = \tilde{\delta}_1 +\tilde{\delta}_2+\tilde{\delta}_3,\ \mu =1+\mu_1,\] for some $\delta_i,\tilde{\delta}_i$, $i\in \{1,2,3\}$ and $\mu_1$ to be determined in the following, we impose that
\begin{equation}
\label{defdelta1}
\left\{\begin{array}{ll}\delta_1 &=g^{-\frac{1}{4}} k-\xi \psi_0\\ \delta_1 u^\prime_1 (r_0)&=2\psi_0 \mu_1 +\frac{n-1}{r_0} \psi_0 \xi + b_0 g^{-1/4}\\ 
 \tilde{\delta}_1 &=g^{-\frac{1}{4}} k+\xi \psi_0  \\
\tilde{\delta}_1 v^\prime_1 (r_0)&=-2\psi_0 \mu_1 -\frac{n-1}{r_0} \psi_0 \xi + b_0 g^{-1/4}.\\ 
\end{array}\right.
\end{equation}

Using the first nondegeneracy assumption in Theorem \ref{thm}, we have that \begin{equation}\label{eqNondegen}
u_1^\prime (r_0)\neq v_1^\prime (r_0).\end{equation} Indeed, if not then the nontrivial function\[z=\left\{\begin{array}{c}
                                                                                                  v_1,\ r\in (0,r_0), \\
                                                                                                  u_1,\ r\in (r_0,1),
                                                                                                \end{array}
\right.\] would be smooth and in the kernel of the linearization of (\ref{eqlimit}) about $w$, which contradicts the aforementioned nondegeneracy assumption.

This last inequality (\ref{eqNondegen}) allows us to solve the previous system. Moreover, we find
$$\mu_1 =O(g^{-\frac{1}{4}}),\ \xi=O(g^{-\frac{1}{4}}),\ \delta_1 =O(g^{-\frac{1}{4}}),\ \tilde{\delta}_1 =O(g^{-\frac{1}{4}}).$$ For future purposes, we note that
\[
\xi=\frac{u_1'(r_0)+v_1'(r_0)}{u_1'(r_0)-v_1'(r_0)}\frac{1}{\psi_0}g^{-\frac{1}{4}}k-\frac{2}{\psi_0}\frac{b_0g^{-\frac{1}{4}}}{u_1'(r_0)-v_1'(r_0)},
\]
\[\begin{array}{ccc}
    \mu_1 & = & -\frac{n-1}{2r_0}\frac{u_1'(r_0)+v_1'(r_0)}{u_1'(r_0)-v_1'(r_0)}\frac{1}{\psi_0}g^{-\frac{1}{4}}k-\frac{u_1'(r_0)v_1'(r_0)}{u_1'(r_0)-v_1'(r_0)}\frac{1}{\psi_0}g^{-\frac{1}{4}}k \\
      &   &  +\frac{n-1}{r_0}\frac{1}{\psi_0}\frac{1}{u_1'(r_0)-v_1'(r_0)}g^{-\frac{1}{4}}b_0+ \frac{u_1'(r_0)+v_1'(r_0)}{2\left(u_1'(r_0)-v_1'(r_0)\right)}\frac{1}{\psi_0}g^{-\frac{1}{4}}b_0,
  \end{array}
\]
\[
\delta_1=- \frac{2g^{-\frac{1}{4}}kv_1'(r_0)}{ u_1'(r_0)-v_1'(r_0) } +\frac{2}{ u_1'(r_0)-v_1'(r_0) }g^{-\frac{1}{4}}b_0,
\]
\[\tilde{\delta}_1= \frac{2g^{-\frac{1}{4}}ku_1'(r_0)}{ u_1'(r_0)-v_1'(r_0) }-\frac{2}{ u_1'(r_0)-v_1'(r_0) }g^{-\frac{1}{4}}b_0.
\]

As previously, we determine $\delta_2,\tilde{\delta}_2,A_0,B_0$ by equating at main order in $s^0$ and $s^1$, the expansions of the inner and outer solutions. We get
$$\begin{cases}\delta_2& = \mu_1 g^{-1/4}k - 2\psi_0 \mu_1 \xi - \frac{n-1}{2r_0}\psi_0 \xi^2 +g^{-1/2} (A_0 \psi_0 +B_0 k)-2B_0 \psi_0 g^{-1/4}\xi  - b_0 g^{-1/4}\xi\\
\delta_2 u_1^\prime(r_0) +\delta_1^2 u_2^\prime(r_0) &=\psi_0 \mu_1^2 +2\mu_1 \frac{n-1}{r_0}\psi_0 \xi + b_0 \mu_1 g^{-1/4} +2 g^{-1/4} B_0 \psi_0\mu_1 \\
&- \dfrac{n-1}{ r_0^2}\xi^2  \psi_0   +( \frac{n-1}{ r_0^2}+(\frac{n-1}{r_0})^2 )  \psi_0 \xi^2 /2 +2g^{-1/4}B_0\psi_0\\
&\ \ \ \ +\frac{n-1}{r_0} \xi  g^{-1/4} (b_0+ 2B_0 \psi_0)+ g^{-1/4} f^\prime (0) (-k \xi + g^{1/4} \psi_0 \xi^2 /2)\\
&\ \ \ \  +g^{-1/2}   b_1    \\
\tilde{\delta}_2& = \mu_1 g^{-1/4}k + 2\psi_0 \mu_1 \xi + \frac{n-1}{2r_0}\psi_0 \xi^2 +g^{-1/2} (-A_0 \psi_0 +B_0 k)+2B_0 \psi_0 g^{-1/4}\xi  - b_0 g^{-1/4}\xi\\
\tilde{\delta}_2 v_1^\prime(r_0) +\tilde{\delta}_1^2 v_2^\prime(r_0) &=-\psi_0 \mu_1^2 -2\mu_1 \frac{n-1}{r_0}\psi_0 \xi + b_0 \mu_1 g^{-1/4} -2 g^{-1/4} B_0 \psi_0 \mu_1\\
&+ \dfrac{n-1}{ r_0^2}\xi^2  \psi_0   -( \frac{n-1}{ r_0^2}+(\frac{n-1}{r_0})^2 )  \psi_0 \xi^2 /2-2g^{-1/4}B_0\psi_0\\
&\ \ \ \ +\frac{n-1}{r_0} \xi  g^{-1/4} (b_0- 2B_0 \psi_0)+ g^{-1/4} f^\prime (0) (-k \xi - g^{1/4} \psi_0 \xi^2 /2)\\
&\ \ \ \  +g^{-1/2}   b_1     .
\end{cases}
$$
One can check that the previous system is solvable thanks to the fact that $u_1^\prime (r_0)\neq v_1^\prime (r_0)$. More precisely, the determinant of the above system   for $A_0,\ B_0,\ \delta_1,\ \tilde{\delta}_1$ is
\[
 2\psi_0\left(u_1'(r_0)-v_1'(r_0) \right)g^{-\frac{1}{4}}+2\left[ \frac{n-1}{2r_0}\left(u_1'(r_0)+v_1'(r_0) \right)k+\left(\frac{u_1'(r_0)+v_1'(r_0)}{2}-\frac{n-1}{r_0} \right)b_0\right]g^{-\frac{1}{2}}.
\] In fact, we obtain that
$$\delta_2= O(g^{-1/2}),\ \tilde{\delta}_2= O(g^{-1/2}),\ A_0=O(1),\ B_0=O(g^{-1/4}).$$

Finally, we determine $\delta_3,\tilde{\delta}_3,A_1,B_1$ by once more equating at main order in $s^0$ and $s^1$ the expansions of the inner and outer solutions. We find
$$\begin{cases}\delta_3 & = - \frac{n-1}{r_0}\psi_0 \mu_1 \xi^2 + \dfrac{n-1}{2 r_0^2}\xi^3  \psi_0    -( \frac{n-1}{ r_0^2}+(\frac{n-1}{r_0})^2 )  \psi_0 \xi^3 /6\\
&-\frac{n-1}{2r_0} \xi^2  g^{-1/4} (b_0+2B_0\psi_0)+ g^{-1/4} f^\prime (0) (k \xi^2 /2 -g^{1/4} \psi_0 \xi^3 /6)\\
&-g^{-1/2} ( b_1 \xi   + 2 B_1 \psi_0 \xi ) + g^{-3/4}(a_1+A_1\psi_0 +B_1 k) \\ &- g^{-1/4}\mu_1 \xi(b_0+2B_0 \psi_0  )-\psi_0\mu_1^2\xi \\
\delta_3 u_1^\prime(r_0) +2\delta_1 \delta_2 u_2^\prime(r_0) +\delta_1^3 u_3^\prime(r_0) &=  \frac{n-1}{r_0}\psi_0 \mu_1^2 \xi     \\
 &-2 \dfrac{n-1}{ r_0^2}\xi^2  \psi_0 \mu_1   +( \frac{n-1}{ r_0^2}+(\frac{n-1}{r_0})^2 ) \mu_1 \psi_0 \xi^2 \\
&\ \ \ \ +\frac{n-1}{r_0}\mu_1 \xi  g^{-1/4} (b_0+ 2B_0 \psi_0)+\mu_1 g^{-1/4} f^\prime (0) (-k \xi + g^{1/4} \psi_0 \xi^2 )\\
&\ \ \ \  +g^{-1/2}\mu_1 ( b_1    + 2  B_1 \psi_0  ) +2B_1\psi_0g^{-1/2}\\
\tilde{\delta}_3 & =  \frac{n-1}{r_0}\psi_0 \mu_1 \xi^2 - \dfrac{n-1}{2 r_0^2}\xi^3  \psi_0    +( \frac{n-1}{ r_0^2}+(\frac{n-1}{r_0})^2 )  \psi_0 \xi^3 /6\\
&-\frac{n-1}{2r_0} \xi^2  g^{-1/4} (b_0-2B_0\psi_0)+ g^{-1/4} f^\prime (0) (k \xi^2 /2 +g^{1/4} \psi_0 \xi^3 /6)\\
&-g^{-1/2} ( b_1 \xi   - 2 B_1 \psi_0 \xi ) + g^{-3/4}(a_1 -A_1\psi_0 +B_1 k) \\&- g^{-1/4}\mu_1 \xi(b_0-2B_0 \psi_0  )+\psi_0\mu_1^2\xi \\
\tilde{\delta}_3 v_1^\prime(r_0) +2\tilde{\delta}_1 \tilde{\delta}_2 v_2^\prime(r_0) + \tilde{\delta}_1^3 v_3^\prime(r_0) &= - \frac{n-1}{r_0}\psi_0 \mu_1^2 \xi     \\
 &+2 \dfrac{n-1}{ r_0^2}\xi^2  \psi_0 \mu_1   -( \frac{n-1}{ r_0^2}+(\frac{n-1}{r_0})^2 ) \mu_1 \psi_0 \xi^2 \\
&\ \ \ \ +\frac{n-1}{r_0}\mu_1 \xi  g^{-1/4} (b_0- 2B_0 \psi_0)+\mu_1 g^{-1/4} f^\prime (0) (-k \xi - g^{1/4} \psi_0 \xi^2 )\\
&\ \ \ \  +g^{-1/2}\mu_1 ( b_1    - 2  B_1 \psi_0  )-2B_1\psi_0g^{-1/2}.
\end{cases}
$$
As previously, since $u_1^\prime (r_0)\neq v_1^\prime (r_0)$, the system is solvable and we get
$$\delta_3= O(g^{-3/4}),\ \tilde{\delta}_3= O(g^{-3/4}),\ A_1=O(1),\ B_1=O(g^{-1/4}).$$

Let us also observe that our choice of parameters allows us to have a matching of the $s^2$ and $s^3$ terms up to order $g^{-1/2}$ and $g^{-1/4}$ respectively in the expansions of the outer and inner solutions. More precisely, using the equation of $u_0$, $u_1$ and the first and second lines of \eqref{defdelta1}, we can rewrite the $s^2$ term of \eqref{outerexp} as
\begin{align*}\frac{u_0^{\prime \prime}(r_0)+\delta u_1^{\prime \prime}(r_0) }{2}&=-\frac{n-1}{2 r_0}\psi_0 +\frac{\delta}{2} (-\frac{n-1}{r_0}u_1^\prime (r_0) +f^\prime (0))\\
&= -\frac{n-1}{2 r_0}\psi_0 +\frac{\delta_1}{2} (-\frac{n-1}{r_0}u_1^\prime (r_0) +f^\prime (0))+ O(g^{-1/2})\\
&=  -\frac{n-1}{2 r_0}\psi_0 - \frac{n-1}{2r_0} (2\psi_0 \mu_1 + \frac{n-1}{r_0}\psi_0 \xi + b_0 g^{-1/4})+ \frac{g^{-1/4}k -\xi \psi_0}{2} f^\prime (0)+ O(g^{-1/2}).
\end{align*}
On the other hand, the term of power $s^2$ in \eqref{innerexp2} + \eqref{innerexp3} can be rewritten as
\begin{align*}&-\frac{n-1}{2r_0}\psi_0 \mu^2+ \dfrac{n-1}{2 r_0^2}\xi  \psi_0 \mu^2   -( \frac{n-1}{ r_0^2}+(\frac{n-1}{r_0})^2 ) \mu^{2} \psi_0 \xi /2\\
& -\frac{n-1}{r_0}\mu   g^{-1/4} (b_0/2 + B_0 \psi_0)+\mu g^{-1/4} f^\prime (0) (k/2  - g^{1/4}\mu \psi_0 \xi /2) \\
 &=-\frac{n-1}{2r_0}\psi_0-\frac{n-1}{r_0}\psi_0 \mu_1   -(\frac{n-1}{r_0})^2  \psi_0 \xi /2\\
& -\frac{n-1}{2r_0}   g^{-1/4} b_0+ g^{-1/4} f^\prime (0) (k/2  - g^{1/4} \psi_0 \xi /2)+ O(g^{-1/2}) .
  \end{align*}
Concerning the terms in $s^3$, we notice differentiating the equation satisfied by $u_0$ that
$$u_0^{\prime \prime \prime}(r_0)/6= \frac{1}{6}\left( \frac{n-1}{r_0^2} \psi_0   +( \frac{n-1}{r_0})^2 \psi_0 +f^\prime (0) \psi_0 \right) .$$
 It is immediate to check that at main order, it is equal to the terms of power $s^3$ in \eqref{innerexp3}.
Analogous computations hold for the $v$ component.

The above analysis brings us to the following proposition, which will be needed in the upcoming subsection.
\begin{prop}\label{proDiffer}
The estimates
\[
|u_{out}(r)-u_{in}(r)|\leq  C|\ln g|^4g^{-1},
\]
\[
 |u_{out}'(r)-u_{in}'(r)|\leq  C|\ln g|^3g^{-\frac{3}{4}},
\]
\[
|u_{out}''(r)-u_{in}''(r)|\leq C|\ln g|^2g^{-\frac{1}{2}},
\]
hold  for $c |\ln g| g^{-\frac{1}{4}}\leq r-r_0 \leq C |\ln g| g^{-\frac{1}{4}}$.

Moreover, for any $c'>0$, we have that the following estimate
\[|u_{in}(r)|+g^{-\frac{1}{4}}|u_{in}'(r)|+g^{-\frac{1}{2}}|u_{in}''(r)|\leq Cg^{-\frac{1}{4}}e^{c'g^{\frac{1}{4}}(r-r_0)}\]
holds for $-C |\ln g| g^{-\frac{1}{4}}\leq r-r_0 \leq -c |\ln g| g^{-\frac{1}{4}}$.

Analogous estimates hold for the difference $v_{out}-v_{in}$.
\end{prop}

\begin{proof}
By the previous analysis, recalling especially (\ref{outerexp}), (\ref{innerexp1}), (\ref{innerexp3}) and (\ref{innerexp4}), we obtain that
\[
u_{out}(r)-u_{in}(r)=O(g^{-1})+O(g^{-1})s+O(g^{-\frac{1}{2}})s^2+O(g^{-\frac{1}{4}})s^3+\sum_{j=0}^{3} {O(g^{-\frac{j}{4}} s^{4-j})},
\]
with $s=r-r_0$, uniformly for $c |\ln g| g^{-\frac{1}{4}}\leq r-r_0 \leq C |\ln g| g^{-\frac{1}{4}}$ (having chosen $c'$ in (\ref{innerexp1}) and (\ref{innerexp4}) sufficiently large). Moreover, the above relation can be differentiated in the obvious fashion. The first assertion of the proposition now follows at once. The second assertion is a direct consequence of the super-exponential decay of $V,\varphi_0,\varphi_1$ with respect to $t\to -\infty$.
\end{proof}
\section{Gluing the outer and inner approximate solutions: The global approximate solution $(u_{ap},v_{ap})$}
We define a global approximate solution to the problem, valid in the whole ball $r\in (0,1)$, by   gluing the inner approximate solution  with the outer one as follows:
\[
(u_{ap},v_{ap})=\left(u_{in}+\zeta (u_{out}-u_{in}),v_{in}+\zeta (v_{out}-v_{in})\right),
\]
where the cutoff function
\[
\zeta(r)=\left\{\begin{array}{cc}
                  0 & \textrm{if}\ |r-r_0|\leq |\ln g|g^{-\frac{1}{4}}, \\
                    &   \\
                  1 & \textrm{if}\ |r-r_0|\geq 2|\ln g|g^{-\frac{1}{4}},
                \end{array}
 \right.
\]
satisfies
\begin{equation}\label{eqcutt}
0\leq \zeta\leq 1,\ |\zeta'|\leq C|\ln g|^{-1}g^\frac{1}{4},\ |\zeta''|\leq C|\ln g|^{-2}g^\frac{1}{2}.
\end{equation}

The following estimates hold.
\begin{prop}\label{proRemainderGlobal}
The remainder
\begin{equation}\label{eqRemainder}
\mathcal{R}=\left(\begin{array}{c}
                    -u_{ap}''-\frac{n-1}{r}u_{ap}'+f(u_{ap})+gv_{ap}^2u_{ap} \\
                     \\
                    -v_{ap}''-\frac{n-1}{r}v_{ap}'+h(v_{ap})+gu_{ap}^2v_{ap}
                  \end{array}
 \right)
\end{equation}
satisfies
\[
\mathcal{R}=O\left(g^{-\frac{1}{2}} \right)\left(\begin{array}{c}
                                                   |\ln g|^4 \\

                                                   e^{-2g^\frac{1}{4}(r-r_0)}
                                                 \end{array}
 \right) \ \textrm{if}\ 0\leq r-r_0\leq 2|\ln g|g^{-\frac{1}{4}},\ \ \mathcal{R}=\left(\begin{array}{c}
                                                   0 \\

                                                   0
                                                 \end{array}
 \right)\  \textrm{if}\ r-r_0\geq 2|\ln g|g^{-\frac{1}{4}},
\]
and the obvious analogue of the above relation holds for $r\leq r_0$.
\end{prop}\begin{proof}
In the inner region $|r-r_0|\leq |\ln g|g^{-\frac{1}{4}}$, the assertion follows from the first assertion of Proposition \ref{propInnerRemain}, while in the outer region $|r-r_0|\geq 2|\ln g|g^{-\frac{1}{4}}$ it follows at once from the definition of the outer approximate solution. It thus remains to consider the intermediate zone $|\ln g|g^{-\frac{1}{4}} \leq |r-r_0|\leq 2|\ln g|g^{-\frac{1}{4}}$. In this intermediate zone and for $r>r_0$ we find that
\begin{equation}\label{eqLong}\begin{array}{rcl}
    -u_{ap}''-\frac{n-1}{r}u_{ap}'+f(u_{ap})+gv_{ap}^2u_{ap} & = & -u_{in}''-\frac{n-1}{r}u_{in}'+f(u_{in})+gv_{in}^2u_{in} \\
      &   &   \\
      &  & +f(u_{ap})-f(u_{in})+gv_{ap}^2u_{ap}-gv_{in}^2u_{in} \\
      & &\\
      &   &  -\zeta''(u_{out}-u_{in})-2\zeta'(u_{out}'-u_{in}')-\zeta (u_{out}''-u_{in}'')\\
      & & \\
      & &-\frac{n-1}{r}\zeta'(u_{out}-u_{in})-\frac{n-1}{r}\zeta(u_{out}'-u_{in}').
  \end{array}\end{equation}

The first line of the righthand side of the above relation is of the desired order by Proposition \ref{propInnerRemain}. Concerning the second line, from Proposition \ref{proDiffer} we obtain that
\[
|f(u_{ap})-f(u_{in})|\leq C|u_{ap}-u_{in}|=C|\zeta(u_{out}-u_{in})|\leq C|\ln g|^4g^{-1}.
\]
Moreover, recalling that $v_{out}=0$  in this region, we find that
\[
gv_{ap}^2u_{ap}-gv_{in}^2u_{in}=g(1-\zeta)^2v_{in}^2u_{ap}-gv_{in}^2u_{in}=gv_{in}^2\left((1-\zeta)^2u_{ap}-u_{in}\right).\]
The above term is much smaller than what we need thanks to the corresponding second assertion of Proposition \ref{proDiffer}. The terms in the last two lines of (\ref{eqLong}) can be estimated as desired by combining the first assertion of Proposition \ref{proDiffer} and (\ref{eqcutt}). On the other hand, the second equation of the remainder can be estimated analogously by making use of Proposition \ref{propInnerRemain} and the corresponding second assertion of Proposition \ref{proDiffer}. The case with $r<0$ in the intermediate zone can be treated identically.
\end{proof}

\section{Linear analysis}\label{seclinear}

In this section we will study the invertibility properties of the linearization of (\ref{eq2comp}) about the approximate solution $(u_{ap},v_{ap})$ between carefully chosen weighted spaces.

The following lemma will be useful in the proof of the upcoming proposition.
\begin{lem}\label{lemBlowUP}
Given any $\rho>0$, there exist  constants $C,\ g_0>0$ such that the following a-priori estimates hold:

If
\[\left\{\begin{array}{l}
    -{\psi}''-\frac{N-1}{r}{\psi}'+gu_{ap}^2{\psi}=h,\ \ r\in \left(1-\rho (\ln g)g^{-\frac{1}{4}},1\right), \\
    \psi\left(1-\rho (\ln g)g^{-\frac{1}{4}}\right)=\psi(1)=0,
  \end{array}\right.
\]
then
\[
\|\psi \|_{L^\infty\left(1-\rho (\ln g)g^{-\frac{1}{4}},1\right)}\leq   C   g^{-\frac{1}{2}}\|h \|_{L^\infty\left(1-\rho (\ln g)g^{-\frac{1}{4}},1\right)},
\]
provided that $g \geq g_0$.

If
\[\left\{\begin{array}{l}
    -{\psi}''-\frac{N-1}{r}{\psi}'+gu_{ap}^2{\psi}=(1-r)^2h,\ \ r\in \left(1-\rho (\ln g)g^{-\frac{1}{4}},1\right), \\
    \psi\left(1-\rho (\ln g)g^{-\frac{1}{4}}\right)=\psi(1)=0,
  \end{array}\right.
\]
then
\[
\|\psi \|_{L^\infty\left(1-\rho (\ln g)g^{-\frac{1}{4}},1\right)}\leq   C   g^{-1}\|h \|_{L^\infty\left(1-\rho (\ln g)g^{-\frac{1}{4}},1\right)},
\]
provided that $g \geq g_0$.
\end{lem}

\begin{proof}
The above a-priori estimates follow from a blow-up argument, using that
\[
u_{ap}^2(r)=u_{0}^2(r)=\left(u_0'(1)\right)^2(1-r)^2+O\left((1-r)^3 \right)\ \ \textrm{as}\ \ r\to 1^-,
\]
keeping in mind that $u_0'(1)<0$ (from Hopf's boundary point lemma).
We refer to \cite[Prop. 3.23]{karaliSourdisAN} for the details.
\end{proof}
Our main result in the section  is the following.

\begin{prop}\label{proLinear}
	Suppose that
\[
\mathcal{L}\left(\begin{array}{c}
\phi\\

\psi
\end{array}\right)=\left(\begin{array}{c}
F\\
H
\end{array}\right),\ r\in (0,1);\ \ \phi'(0)=\phi(1)=0,\ \psi'(0)=\psi(1)=0,
\]
where $F,H\in C[0,1]$ and
	\[
	\mathcal{L}\left(\begin{array}{c}
	\phi\\
	\\
	\psi
	\end{array}\right)\equiv\left(\begin{array}{c}
	-\phi''-\frac{N-1}{r}\phi'+f'(u_{ap})\phi+gv_{ap}^2\phi+2gu_{ap}v_{ap}\psi\\
	\\
	-\psi''-\frac{N-1}{r}\psi'+h'(v_{ap})\psi+gu_{ap}^2\psi+2gu_{ap}v_{ap}\phi
	\end{array}\right).
	\]
Then, given $\gamma \in (0,1)$, there exist $C,g_0>0$, independent of $(F,H)$ and $(\phi,\psi)$, such that
\[
\| (\phi,\psi)\|_1\leq Cg^{-\frac{1}{4}}\|(F,H)\|_2,
\] 	
where
\begin{equation}\label{eqPnorm0}
\|(\Phi,\Psi) \|_i=\|w_i(r-r_0)\Phi\|_{L^\infty(0,1)}+\|w_i(r_0-r)\Psi\|_{L^\infty(0,1)}, \ \ i=1,2,
\end{equation}
with
\[
w_1(s)=\left\{\begin{array}{ll}
1,& s\geq 0,\\
\\
e^{g^\frac{1}{4}|s|},&s<0,
\end{array}\right.
\ \ \
w_2(s)=\left\{\begin{array}{ll}
1+|g^\frac{1}{4}s|^{1+\gamma},& s\geq 0,\\
\\
e^{g^\frac{1}{4}|s|},&s<0,
\end{array}\right.
\]
provided that $g\geq g_0$.
\end{prop}
\begin{proof}
Suppose to the contrary that there exist $g_n\to +\infty$ and pairs $(F_n,H_n)\in C[0,1]$, $(\phi_n,\psi_n)\in C^2(0,1)\cap C[0,1]$ such that	
\begin{equation}\label{eqPtest}
\mathcal{L}\left(\begin{array}{c}
\phi_n\\

\psi_n
\end{array}\right)=\left(\begin{array}{c}
F_n\\
H_n
\end{array}\right);\ \ \phi_n'(0)=\phi_n(1)=0,\ \psi_n'(0)=\psi_n(1)=0,
\end{equation}
while
\begin{equation}\label{eqPnormaliz}
\|(\phi_n,\psi_n) \|_1=1,\ \ g_n^{-\frac{1}{4}}\|(F_n,H_n) \|_2\to 0.
\end{equation}

Observe that there exists $C>0$ such that
\[
gu_{ap}^2\geq 4 g^{\frac{1}{2}},\ \ r_0+Cg^{-\frac{1}{4}}\leq r\leq 1-Cg^{-\frac{1}{4}},
\]
for $g$ sufficiently large (keep in mind that $u_0'(1)<0$ from Hopf's boundary point lemma). Furthermore, it holds
\begin{equation}\label{eqPexpGG}gu_{ap}v_{ap}\leq Cg^\frac{1}{2}e^{-4g^\frac{1}{4}|r-r_0|},\ \ r\in [0,1].\end{equation}
Hence, a standard barrier argument yields that
\begin{equation}\label{eqPexp}\begin{split}
|\psi_n(x)|\leq &|\psi_n(r_0+Cg_n^{-\frac{1}{4}})|e^{-2g_n^\frac{1}{4}(r-r_0-Cg_n^{-\frac{1}{4}})} +o(g_n^{-\frac{1}{4}})e^{g_n^\frac{1}{4}(r_0-r)}\\ &+C\|\phi_n\|_{L^\infty(r_0,1)}e^{2g_n^\frac{1}{4}(r_0-r)}+|\psi_n(1-Cg_n^{-\frac{1}{4}})|e^{2g_n^\frac{1}{4}(r-1+Cg_n^{-\frac{1}{4}})}\end{split}
\end{equation}
for $r\in (r_0+Cg_n^{-\frac{1}{4}},1-Cg_n^{-\frac{1}{4}})$.
In particular, recalling (\ref{eqPnormaliz}), for any $\rho>0$, we have
\begin{equation}\label{eqPlog}
\left|\psi_n\left(1-\rho (\ln g_n)g_n^{-\frac{1}{4}}\right)\right|\leq Ce^{g_n^\frac{1}{4}(r_0-1)}
\left(g_n^{2\rho}e^{g_n^\frac{1}{4}(r_0-1)}+o(g_n^{\rho-\frac{1}{4}})+g_n^{-2\rho} \right),
\end{equation}
as $n\to \infty$.
Now, let us write
\begin{equation}\label{eqPwrite}
\psi_n(r)=\psi_n\left(1-\rho (\ln g_n)g_n^{-\frac{1}{4}}\right)R_n(r)+\hat{\psi}_n(r),\ \ r\in \left(1-\rho (\ln g_n)g_n^{-\frac{1}{4}},1\right),
\end{equation}where
\[R_n(r)=\left\{\begin{array}{ll}
    \frac{1-r^{2-N}}{1-\left(1-\rho (\ln g_n)g_n^{-\frac{1}{4}}\right)^{2-N}} & \textrm{if}\ N\geq 3, \\
      &   \\
    \frac{\ln r}{\ln \left(1-\rho (\ln g_n)g_n^{-\frac{1}{4}}\right)} & \textrm{if}\ N= 2.
  \end{array}\right.
\]
Then, from (\ref{eqPtest}), (\ref{eqPnormaliz}), (\ref{eqPexpGG}) and the fact that $u_{ap}\leq C(1-r)$ in $(1-d,1)$ for some small $d>0$ independent of $g$, it holds
\[
-\hat{\psi}_n''-\frac{N-1}{r}\hat{\psi}_n'+g_nu_{ap}^2\hat{\psi}_n=\hat{H}_n
\]
with
\[|\hat{H}_n|\leq Cg_n(1-r)^2
\left|\psi_n\left(1-\rho (\ln g_n)g_n^{-\frac{1}{4}}\right)\right|+Cg_n^{\frac{1}{2}+4\rho}e^{4g_n^{\frac{1}{4}}(r_0-1)}
+o\left(g_n^{\frac{1}{4}+\rho}\right)e^{g_n^{\frac{1}{4}}(r_0-1)}
\]
for $r\in \left(1-\rho (\ln g_n)g_n^{-\frac{1}{4}},1\right)$.
Moreover, we have
\[
\hat{\psi}_n\left(1-\rho (\ln g_n)g_n^{-\frac{1}{4}}\right)=\hat{\psi}_n\left(1\right)=0.
\]
Hence,  we deduce by Lemma \ref{lemBlowUP} that
\[\|\hat{\psi}_n\|_{L^\infty\left(1-\rho (\ln g_n)g_n^{-\frac{1}{4}},1\right)}\leq C
\left|\psi_n\left(1-\rho (\ln g_n)g_n^{-\frac{1}{4}}\right)\right|
+o\left(g_n^{-\frac{1}{4}+\rho}\right)e^{g_n^{\frac{1}{4}}(r_0-1)},
\]
as $n\to \infty$.
Consequently, via (\ref{eqPlog}) and (\ref{eqPwrite}), we get
\begin{equation}\label{eqPlogBig}\|{\psi}_n\|_{L^\infty\left(1-\rho (\ln g_n)g_n^{-\frac{1}{4}},1\right)}\leq
C e^{g_n^\frac{1}{4}(r_0-1)}
\left(g_n^{2\rho}e^{g_n^\frac{1}{4}(r_0-1)}+o(g_n^{\rho-\frac{1}{4}})+g_n^{-2\rho} \right),
\end{equation}
as $n\to \infty$.

In light of (\ref{eqPtest}), (\ref{eqPnormaliz}), (\ref{eqPexpGG}) together with the exponential decay estimates for $u_{ap},v_{ap}$ that led to it, standard elliptic estimates and the usual diagonal-compactness argument,
passing to a further subsequence if needed, we find that
\begin{equation}\label{eqPCloc1}
\phi_n\to \phi_\infty\ \ \textrm{in}\ \ C^1_{loc}(r_0,1],
\end{equation}
where $\phi_\infty$ satisfies
\[
-\phi_\infty''-\frac{N-1}{r}\phi_\infty'+f'(u_0)\phi_\infty=0\ \textrm{in}\ (r_0,1);\ \phi_\infty(1)=0.
\]
In other words, we have
\begin{equation}\label{eqPCloc2}
\phi_\infty = \lambda_+u_1,\ r\in [r_0,1]\  \textrm{for some}\ \lambda_+\in \mathbb{R}.
\end{equation}
Analogously, passing to a further subsequence if needed, it holds
\begin{equation}\label{eqPCloc3}
\psi_n\to \psi_\infty\equiv \lambda_- v_1 \ \ \textrm{in}\ \ C_{loc}[0,r_0)\cap C_{loc}^1(0,r_0)\  \textrm{for some}\ \lambda_-\in \mathbb{R}.
\end{equation}
We will next show that
\begin{equation}\label{eqPbalance}
\psi_\infty(r_0)+\phi_\infty(r_0)=0\ \textrm{and}\ \psi'_\infty(r_0)+\phi'_\infty(r_0)=0.
\end{equation}
To this end, we
need some preliminary observations.

Firstly, recalling that $t=\mu g^\frac{1}{4}(r-r_0-\xi)$, we note that the following estimates hold:
\begin{equation}\label{eqP1}
u_{ap}^2(r)=\mu^2g^{-\frac{1}{2}}U^2(t)+g^{-\frac{3}{4}}p_{22}(t)+\left\{\begin{array}{ll}
                                                                   \sum_{k=0}^4O(g^{-\frac{k}{4}})(r-r_0-\xi)^{4-k},& r\in [r_0,1], \\
                                                                   O(g^{-1})e^{-cg^\frac{1}{4}|r-r_0-\xi|},& r\in [0,r_0],
                                                                   \end{array}
 \right.
\end{equation}
\begin{equation}\label{eqP1}
v_{ap}^2(r)=\mu^2g^{-\frac{1}{2}}V^2(t)+g^{-\frac{3}{4}}p_{11}(t)+\left\{\begin{array}{ll}
                                                                   O(g^{-1})e^{-cg^\frac{1}{4}|r-r_0-\xi|},& r\in [r_0,1], \\
                                                                   \sum_{k=0}^4O(g^{-\frac{k}{4}})(r-r_0-\xi)^{4-k},& r\in [0,r_0],
                                                                   \end{array}
 \right.
\end{equation}

\begin{equation}\label{eqP2}
u_{ap}v_{ap}(r)=\mu^2g^{-\frac{1}{2}}UV(t)+g^{-\frac{3}{4}}p_{12}(t)
                                                                   +O(g^{-1})e^{-cg^\frac{1}{4}|r-r_0-\xi|},\ \ r\in [0,1],
                                                                   \end{equation}where the above functions $p_{ij}$ are independent of $g$ and satisfy
                                                                   \[|p_{22}(t)|\leq \left\{\begin{array}{ll}
                                                                                           C(t^3+1), & t\geq 0, \\
                                                                                             &   \\
                                                                                           Ce^{ct}, & t\leq 0,
                                                                                         \end{array}
                                                                   \right.
                                                                   \]
\[|p_{11}(t)|\leq C |p_{22}(-t)|,\ \ |p_{12}(t)|\leq Ce^{-c|t|},\ \ t\in \mathbb{R}.\]
We point out that the above estimates hold uniformly as $g\to +\infty$.

 To obtain information near $r_0$, we will employ
a blow-up argument. To this end, let
\begin{equation}\label{eqPhi}
{\Phi}(t)=\phi\left(r_0+\xi+\mu^{-1}g^{-\frac{1}{4}}t\right),\ \ {\Psi}(t)=\psi\left(r_0+\xi+\mu^{-1}g^{-\frac{1}{4}}t\right)
\end{equation}
for $\mu g^\frac{1}{4}(-r_0-\xi)\leq t \leq \mu g^\frac{1}{4}(1-r_0-\xi)$.
We find readily from (\ref{eqPtest}) that
\[
\begin{array}{c}
	-\Phi''-\frac{N-1}{r_0+\xi+\mu^{-1}g^{-\frac{1}{4}}t}\mu^{-1}g^{-\frac{1}{4}}\Phi'+\mu^{-2}g^{-\frac{1}{2}}f'(u_{ap})\Phi+\mu^{-2}g^{\frac{1}{2}}v_{ap}^2\Phi+2\mu^{-2}g^{\frac{1}{2}}u_{ap}v_{ap}\Psi=\mu^{-2}g^{-\frac{1}{2}}F,\\
	\\
	-\Psi''-\frac{N-1}{r_0+\xi+\mu^{-1}g^{-\frac{1}{4}}t}\mu^{-1}g^{-\frac{1}{4}}\Psi'+\mu^{-2}g^{-\frac{1}{2}}h'(v_{ap})\Psi+\mu^{-2}g^{\frac{1}{2}}u_{ap}^2\Psi+2\mu^{-2}g^{\frac{1}{2}}u_{ap}v_{ap}\Phi=\mu^{-2}g^{-\frac{1}{2}}H,
	\end{array}
\]
where $u_{ap},v_{ap},F,H$ are evaluated at $r_0+\xi+\mu^{-1}g^{-\frac{1}{4}}t$. By   (\ref{eqPnormaliz}), (\ref{eqP1})-(\ref{eqP2}), and the standard compactness-diagonal argument, passing to a further subsequence if needed (still denoted by $(\Phi_n,\Psi_n)$), we may assume that the pair $(\Phi_n,\Psi_n)$ converges in $C^1_{loc}(\mathbb{R})$ to a bounded element of the kernel of the linear operator $L$ in (\ref{eqL}) (the trivial limit not being excluded). By virtue of \cite[Prop 5.1]{berestycki1},
the only such elements in the kernel of $L$ are constant multiples of $(U',V')$. Consequently, we have
\begin{equation}\label{eqPlocIn}
(\Phi_n,\Psi_n)\to \sigma(U',V')\ \ \textrm{in}\ C^1_{loc}(\mathbb{R})\ \textrm{as}\ n\to \infty,\ \textrm{for some}\ \sigma \in \mathbb{R}.
\end{equation}

The next observation concerns the perturbed linear operator
\begin{equation}\label{eqPLg}
\texttt{L}_g=L+  \mu^{-2}g^{-\frac{1}{4}}\left(\begin{array}{cc}
                                                                              p_{11}(t) & p_{12}(t) \\
                                                                              p_{12}(t) & p_{22}(t)
                                                                            \end{array}
 \right)I,
\end{equation}
where $L$ is as in (\ref{eqL}). Since the pairs $(U',V')$ and $(tU'+U,tV'+V)$ belong in the kernel of $L$, it is natural to seek corresponding elements in the kernel of $\texttt{L}_g$ in the form
\begin{equation}\label{eqPdecompLg}
\left(\begin{array}{c}
        \Phi_1 \\
        \Psi_1
      \end{array}
 \right)
 =
\left(\begin{array}{c}
        U' \\
        V'
      \end{array}
 \right)+ \mu^{-2}g^{-\frac{1}{4}}\left(\begin{array}{c}
        \hat{\Phi}_1 \\
        \hat{\Psi}_1
      \end{array}
 \right)\ \textrm{and}\ \left(\begin{array}{c}
        \Phi_2 \\
        \Psi_2
      \end{array}
 \right)
 =
\left(\begin{array}{c}
        tU'+U \\
        tV'+V
      \end{array}
 \right)+ \mu^{-2}g^{-\frac{1}{4}}\left(\begin{array}{c}
        \hat{\Phi}_2 \\
        \hat{\Psi}_2
      \end{array}
 \right),
\end{equation}
respectively. It follows readily that the following inhomogeneous problems must be satisfied:
\[
L\left(\begin{array}{c}
        \hat{\Phi}_i\\
        \hat{\Psi}_i
      \end{array}
 \right)+  \mu^{-2}g^{-\frac{1}{4}}\left(\begin{array}{cc}
                                                                              p_{11}(t) & p_{12}(t) \\
                                                                              p_{12}(t) & p_{22}(t)
                                                                            \end{array}
 \right)\left(\begin{array}{c}
        \hat{\Phi}_i\\
        \hat{\Psi}_i
      \end{array}
 \right)=\left(\begin{array}{c}
        \hat{F}_i \\
        \hat{G}_i
      \end{array}
 \right),\ \ i=1,2,
\]
for some $\hat{F}_i,\hat{G}_i$ that are independent of $g$ and tend to zero  super-exponentially fast  as $|t|\to +\infty$.
For our purposes, however, it will be enough to make  the following choices:
\[
L\left(\begin{array}{c}
        \hat{\Phi}_i\\
        \hat{\Psi}_i
      \end{array}
 \right)=\left(\begin{array}{c}
        \hat{F}_i \\
        \hat{G}_i
      \end{array}
 \right),\ \ i=1,2.
\]
By Proposition \ref{proSymA}, the above problems  admit solutions
such that
\begin{equation}\label{eqPass1}
 \hat{\Phi}_i(t)-(a_i)_+-b_it\to 0,\ \hat{\Psi}_i(t)\to 0 \ \textrm{as}\ t\to +\infty,
\end{equation}
\begin{equation}\label{eqPass2}
\hat{\Phi}_i(t)\to 0 ,\ \hat{\Psi}_i(t)-(a_i)_--b_it\to 0 \ \textrm{as}\ t\to -\infty,
\end{equation}
for some $(a_i)_\pm,\ b_i\in \mathbb{R}$, $i=1,2$, super-exponentially fast (we insist on the fact that these solutions do not depend on $g$).
Moreover, analogous estimates hold for their derivatives. Taking into account the asymptotic behavior of $(U,V)$ and its derivatives,    the above choices of $(\hat{\Phi}_i,\hat{\Psi}_i)$ yield that the pairs in (\ref{eqPdecompLg}) satisfy
\begin{equation}\label{eqPremain}
\texttt{L}_g\left(\begin{array}{c}
                    \Phi_i \\

                    \Psi_i
                  \end{array}
 \right)=\mu^{-4} g^{-\frac{1}{2}}\left(\begin{array}{c}
                                          F_i \\
                                          G_i
                                        \end{array}
 \right),\ i=1,2,
\end{equation}
uniformly in $\mathbb{R}$, as $g\to +\infty$, for some fixed functions $F_i,G_i$ that tend to zero super-exponentially fast as $t\to \pm \infty$.

In the sequel, for notational simplicity, we will frequently drop the subscripts $n$. Let $\delta>0$ independent of $n$ and denote $T^\pm=\mu(\pm\delta-\xi) g^\frac{1}{4}$. Integrating by parts, we observe that
\begin{align*}
&  \int_{r_0-\delta}^{r_0+\delta} -\Delta \phi (r) \Phi_i (t) r^{N-1}dr\\
& = \mu^2 g^{1/2} \int_{r_0-\delta}^{r_0+\delta} -\Delta \Phi_i (t) \phi (r) r^{N-1}dr\\
&+\mu g^{1/4}((r_0+\delta)^{N-1} \phi (r_0 +\delta ) \Phi_i^\prime (T_+) - (r_0 -\delta)^{N-1}\phi (r_0 -\delta) \Phi_i^\prime (T_-) )\\
& - ((r_0+\delta)^{N-1} \phi^\prime (r_0 +\delta ) \Phi_i (T^+) - (r_0-\delta)^{N-1} \phi^\prime (r_0 -\delta ) \Phi_i (T^-) ).
\end{align*}
So multiplying (\ref{eqPtest}) by $\left(r^{N-1}\Phi_i\left(t\right),r^{N-1}\Psi_i\left(t\right)\right)$, $i=1,2$, rearranging terms and using the previous identity, we get:
\begin{equation}\label{eqPBig}\begin{array}{c}\mu^2 g^{\frac{1}{2}}\left<r^{N-1}\texttt{L}_g\left(\begin{array}{c}
                    \Phi_i(t) \\

                    \Psi_i(t)
                  \end{array}
 \right), \left(\begin{array}{c}
                    \phi(r) \\

                    \psi(r)
                  \end{array}
 \right) \right>_{L^2\times L^2(r_0-\delta,r_0+\delta)}\\ \\-(N-1)\mu g^\frac{1}{4}\int_{r_0-\delta}^{r_0+\delta}\left(\Phi_i'(t)\phi(r)+\Psi_i'(t)\psi(r) \right)r^{N-2}dr\\ \\
    -(r_0+\delta)^{N-1}\phi'(r_0+\delta)\Phi_i\left(T^+\right)+(r_0-\delta)^{N-1}\phi'(r_0-\delta)\Phi_i\left(T^-\right)
\\ \\+(r_0+\delta)^{N-1}\phi(r_0+\delta)\mu g^\frac{1}{4}\Phi_i'(T^+)-(r_0-\delta)^{N-1}\phi(r_0-\delta)\mu g^\frac{1}{4}\Phi_i'(T^-) \\
      \\
    -(r_0+\delta)^{N-1}\psi'(r_0+\delta)\Psi_i(T^+)+(r_0-\delta)^{N-1}\psi'(r_0-\delta)\Psi_i(T^-)\\ \\
+(r_0+\delta)^{N-1}\psi(r_0+\delta)\mu g^\frac{1}{4}\Psi_i'(T^+)-(r_0-\delta)^{N-1}\psi(r_0-\delta)\mu g^\frac{1}{4}\Psi_i'(T^-) \\
      \\
    +2\int_{r_0-\delta}^{r_0+\delta}\left[gu_{ap}v_{ap}(r)-\mu^2g^{\frac{1}{2}}UV(t)-g^\frac{1}{4}p_{12}(t)\right]\left[\psi(r) \Phi_i(t)+\phi(r) \Psi_i(t)\right]r^{N-1}dr \\
      \\
    +\int_{r_0-\delta}^{r_0+\delta}\left[g v_{ap}^2(r)-\mu^2g^{\frac{1}{2}}V^2(t)-g^\frac{1}{4}p_{11}(t)\right]
\phi(r) \Phi_i(t)r^{N-1}dr \\
      \\
    +\int_{r_0-\delta}^{r_0+\delta}\left[gu_{ap}^2(r)-\mu^2g^{\frac{1}{2}}U^2(t)-g^\frac{1}{4}p_{22}(t)\right]
\psi(r) \Psi_i(t)r^{N-1}dr \\
      \\
    +\int_{r_0-\delta}^{r_0+\delta}\left[f'(u_{ap})\phi(r)\Phi_i(t)+h'(v_{ap})
\psi(r)\Psi_i(t)\right]r^{N-1}dr \\
      \\
    = \\
      \\
   \int_{r_0-\delta}^{r_0+\delta}\left[F(r)\Phi_i(t)+H(r)
\Psi_i(t)\right]r^{N-1}dr.
  \end{array}
\end{equation}

The first term in the above relation, thanks to (\ref{eqPnormaliz}) and (\ref{eqPremain}), can be plainly estimated as follows:
\[
\mu^2 g^{\frac{1}{2}}\left<r^{N-1}\texttt{L}_g\left(\begin{array}{c}
                    \Phi_i(t) \\

                    \Psi_i(t)
                  \end{array}
 \right), \left(\begin{array}{c}
                    \phi(r) \\

                    \psi(r)
                  \end{array}
 \right) \right>_{L^2\times L^2(r_0-\delta,r_0+\delta)}=O(\delta),\ \ i=1,2,
\]
uniformly in $n\geq 1$, as $\delta \to 0$.

Concerning the second line in (\ref{eqPBig}), we will treat the cases $i=1$ and $i=2$ separately. We first note from (\ref{eqPnormaliz}), (\ref{eqPdecompLg}), (\ref{eqPass1}) and (\ref{eqPass2}) that
\[
\mu g^\frac{1}{4}\int_{r_0-\delta}^{r_0+\delta}\left(\Phi_1'(t)\phi(r)+\Psi_1'(t)\psi(r) \right)r^{N-2}dr=\mu g^\frac{1}{4}\int_{r_0-\delta}^{r_0+\delta}\left(U''(t)\phi(r)+V''(t)\psi(r) \right)r^{N-2}dr+O(\delta),
\]
uniformly in $n\geq 1$, as $\delta \to 0$.
By (\ref{eqPhi}), (\ref{eqPlocIn}), the exponential decay of $U'',V''$ as $t \to \pm \infty$ and Lebesgue's dominated convergence theorem, we infer that   the righthand side of the above relation is equal to
\[\begin{split}
\int_{\mu(-\delta-\xi)g^\frac{1}{4}}^{\mu(\delta-\xi)g^\frac{1}{4}}\left(U''(t)\Phi(t)+V''(t)\Psi(t) \right)(r_0+\xi+\mu^{-1}g^{-\frac{1}{4}}t)^{N-2}dt+O(\delta)=
\ \ \ \ \ \ \ \ \ \ \ \ \ \ \ \ \ \ \ \ \ \ \ \ \ \ \ \
\\
\ \ \ \ \ \ \ \ \ \ \ \ \ \ \ \ \ \ \ \ \ \ \ \ \int_{-\infty}^{+\infty}\sigma\left(U''(t)U'(t)+V''(t)V'(t) \right)r_0^{N-2}dt+o(1)+O(\delta)\stackrel{(\ref{eqRefUV})}{=}o(1)+O(\delta),&
\end{split}
\]
as $n\to +\infty$ for each $\delta <1$.
Hence, we have that
\[
\mu g^\frac{1}{4}\int_{r_0-\delta}^{r_0+\delta}\left(\Phi_1'(t)\phi(r)+\Psi_1'(t)\psi(r) \right)r^{N-2}dr=
o(1)+O(\delta)
\]
as $n\to +\infty$ for each $\delta <1$.
 On the other hand, it turns out that for $i=2$ the following rough estimate will be sufficient:
\[
\mu g^\frac{1}{4}\int_{r_0-\delta}^{r_0+\delta}\left(\Phi_2'(t)\phi(r)+\Psi_2'(t)\psi(r) \right)r^{N-2}dr=  g^\frac{1}{4}O(\delta),
\]
uniformly in $n$, as $\delta \to 0$.

We will next show that the terms in the   seventh, eighth and ninth line of (\ref{eqPBig}) tend to zero as $n\to \infty$.
Indeed, by (\ref{eqPnormaliz}), (\ref{eqP1}), (\ref{eqPdecompLg}), (\ref{eqPass1}) and (\ref{eqPass2}), we get
\[
 \begin{split}
    &\left|\int_{r_0}^{r_0+\delta}\left[gu_{ap}^2(r)-\mu^2g^{\frac{1}{2}}U^2(t)-g^\frac{1}{4}p_{22}(t)\right]
\psi(r) \Psi_i(t)r^{N-1}dr\right| \\
    \leq & Cg\|(\phi,\psi)\|_1 \sum_{k=0}^4 g^{-\frac{k}{4}}\int_{r_0}^{r_0+\delta}|r-r_0-\xi|^{4-k}e^{-cg^\frac{1}{4}|r-r_0-\xi|}dr\\
    \leq & C \sum_{k=0}^4 \int_{0}^{\delta}|g^\frac{1}{4}(s-\xi)|^{4-k}e^{-cg^\frac{1}{4}|s-\xi|}ds\\
    \leq& C g^{-\frac{1}{4}},
 \end{split}
\]
where, here and in the sequel, the generic constants $c,C$ are independent of both $n$ and $\delta$. The remaining terms can be handled analogously.

The term in the tenth line of (\ref{eqPBig}) can be estimated in a simple way as follows. Using (\ref{eqPnormaliz}) and (\ref{eqPdecompLg}), we obtain that
\[
\int_{r_0-\delta}^{r_0+\delta}\left[f'\left(u_{ap}\right)\phi(r)\Phi_i(t)+h'\left(v_{ap}\right)
\psi(r)\Psi_i(t)\right]r^{N-1}dr=
O(\delta)g^{\frac{i-1}{4}},\ \ i=1,2,
\]
uniformly in $n\geq 1$, as $\delta \to 0$.

Concerning the last  line of (\ref{eqPBig}), via (\ref{eqPnormaliz}), (\ref{eqPdecompLg}), (\ref{eqPass1}) and (\ref{eqPass2}), we get
\[
\begin{split}
\int_{r_0-\delta}^{r_0+\delta}\left[F(r)\Phi_i(t)+H(r)
\Psi_i(t)\right]r^{N-1}dr=  \ \ \ \ \ \ \ \ \ \ \ \ \ \ \ \ \ \ \ \ \ \ \ \ \ \ \ \ \ \ \ \ \ \ \ \ \ \ \ \ \ \ \ \ \ \ \ \ \ \ \ \ \ \ \ \ \ \ \ & \\ o(1)g^\frac{1}{4}g^\frac{i-1}{4}\int_{-\delta}^{\delta}\frac{1}{1+|g^\frac{1}{4}s|^{1+\gamma}}ds=o(1)g^\frac{i-1}{4}\int_{-\delta g^\frac{1}{4}}^{\delta g^\frac{1}{4}}\frac{1}{1+|\tau|^{1+\gamma}}d\tau=o(1)g^\frac{i-1}{4} \end{split} \]
as $n\to \infty$, $i=1,2$. We point out that the assumption $\gamma>0$ was used crucially in the last equality.

Armed with the above information, recalling  (\ref{eqPCloc1}), (\ref{eqPCloc3}), (\ref{eqPdecompLg}), and the definition of $T^\pm $, passing to a further subsequence if needed, we can let $n\to \infty$ in (\ref{eqPBig}) for $i=1,2$ to arrive, respectively,  at the following two relations:
 \[
     -\phi_\infty'(r_0+\delta)-\psi_\infty'(r_0-\delta) =O(\delta),
 \]
 \[
    - \delta\phi'_\infty(r_0+\delta)
+\phi_\infty(r_0+\delta) +\delta\psi_\infty'(r_0-\delta)+\psi_\infty(r_0-\delta)=O(\delta),
 \]
  as  $\delta \to 0$.
 The desired equalities in (\ref{eqPbalance}) now follow at once by letting $\delta \to 0$ in the above two relations.
Then, in light of  (\ref{eqPCloc2}) and (\ref{eqPCloc3}) and the nondegeneracy assumption  $u_1'(r_0)\neq v_1'(r_0)$ (recall (\ref{eqNondegen})), we get that $\lambda_\pm=0$. Consequently, recalling also (\ref{eqPexp}), we have shown so far that
 \begin{equation}\label{eqPmind}
 \phi_n,\ \psi_n \to 0\ \ \textrm{in}\ C_{loc}\left([0,1]\setminus \{r_0\} \right)\cap C^1_{loc}\left((0,1]\setminus \{r_0\} \right).
 \end{equation}

 We will next extend as much as possible towards $r_0$ the domain of validity of the above relation. Recalling (\ref{eqPnormaliz}), (\ref{eqPexpGG}) and noting that the same exponential decay estimate also holds for $gv_{ap}^2$ in $(r_0,1)$, we find from the first equation of (\ref{eqPtest}) that
 \[
(r^{N-1} \phi')'=r^{N-1}f'(u_{ap})\phi+O(1)g^\frac{1}{2}e^{-cg^\frac{1}{4}(r-r_0)}+o(1)\frac{g^\frac{1}{4}}{1+|g^\frac{1}{4}(r-r_0)|^{1+\gamma}},\ \ r\in (r_0,1),\ \textrm{as}\ n\to \infty.
 \]
 Integrating over $(r,1)$ with $r\in (r_0,1)$,   using that $\phi'(1)=o(1)$ and $\|\phi\|_{L^1(0,1)}=o(1)$ as $n\to \infty$ (from (\ref{eqPnormaliz}), (\ref{eqPmind})), yields
\[\begin{array}{rcl}
    \phi'(r) & = & o(1)+O(1)g^\frac{1}{4}e^{-cg^\frac{1}{4}(r-r_0)}+o(1)\int_{g^\frac{1}{4}(r-r_0)}^{g^\frac{1}{4}(1-r_0)}\frac{1}{1+\tau^{1+\gamma}}d\tau \\
      &   &   \\
      &=   & o(1)+O(1)g^\frac{1}{4}e^{-cg^\frac{1}{4}(r-r_0)},
  \end{array}
   \]
   uniformly in $r\in (r_0,1)$, as $n\to \infty$.
  Similarly, integrating once more gives
   \begin{equation}\label{eqPDecayOut}
   \phi(r) = o(1)+O(1)e^{-cg^\frac{1}{4}(r-r_0)}, \ \ \textrm{uniformly in}\  r\in (r_0,1),\ \textrm{as}\ n\to \infty.
   \end{equation}
  Clearly, there is an analogous relation for $\psi$ in $(d,r_0)$ for some small fixed $d>0$.

Let $M>0$ be independent of $n$.
By putting together the information supplied by (\ref{eqPDecayOut}) for $x=r_0+\xi+\mu^{-1}g^{-\frac{1}{4}}M$ with that from the above relation for $t=M$, via (\ref{eqPhi}), we deduce that
\[
o(1)+O(1)e^{-cM}=\sigma U'(M)+o(1)\ \ \textrm{as}\ n\to \infty,
\]
where the constant $c>0$ is   independent of $M$ as well. Sending first $n\to \infty$ in the above relation, and subsequently $M\to +\infty$ in the resulting one,
we get that $\sigma=0$ (recall that $U'(+\infty)=\psi_0>0$). So, in light of this, relation (\ref{eqPDecayOut}), the analogous one for $\psi_n$, and (\ref{eqPlocIn}) give that
\[
(\phi_n,\psi_n)\to (0,0),\ \textrm{uniformly on}\ [0,1],\ \textrm{as}\ n\to \infty.
\]

Finally, taking into account (\ref{eqPexp}), (\ref{eqPlogBig}), say with $\rho=\frac{1}{4}$,   and the above relation, we arrive at
\[
\|(\phi_n,\psi_n)\|_1\to 0 \ \ \textrm{as}\ n\to \infty,
\]
which contradicts (\ref{eqPnormaliz}).
\end{proof}
\begin{rmq}\label{remP}
The exact same proof of Proposition \ref{proLinear}, without the paragraph leading to (\ref{eqPlogBig}), also yields the still useful estimate
\[
\| (\phi,\psi)\|_{L^\infty \times L^\infty(0,1)}\leq Cg^{-\frac{1}{4}}\|(F,H)\|_0,
\] 	
where the latter norm is defined through
 (\ref{eqPnorm0}) with $i=0$
 and
 \[
w_0(s)=\left\{\begin{array}{ll}
1+|g^\frac{1}{4}s|^{1+\gamma},& s\geq 0,\\
\\
1,&s<0.
\end{array}\right.
\]
Then, analogously to (\ref{eqPexp}), we obtain that
\begin{equation}\label{eqPremark}
|\psi(r)|\leq  Cg^{-\frac{1}{4}}\|(F,H)\|_0e^{2g^\frac{1}{4}(r_0-r)} +Cg^{-\frac{1}{2}}\|(F,H)\|_2e^{g^\frac{1}{4}(r_0-r)}+|\psi(1-Cg^{-\frac{1}{4}})|e^{2g^\frac{1}{4}(r-1+Cg^{-\frac{1}{4}})}
\end{equation}
for \[r\in I=(r_0+Cg^{-\frac{1}{4}},1-Cg^{-\frac{1}{4}}).\]
In turn,  in analogy to (\ref{eqPlog}), it holds
\begin{equation}\label{eqPvia}\begin{split}\left|\psi\left(1-\rho (\ln g)g^{-\frac{1}{4}}\right)\right|\leq &  Cg^{2\rho-\frac{1}{4}}e^{2g^\frac{1}{4}(r_0-1)}\|(F,H)\|_0+Cg^{\rho-\frac{1}{2}}e^{g^\frac{1}{4}(r_0-1)}\|(F,H)\|_2\\&+Cg^{-2\rho}e^{g^\frac{1}{4}(r_0-1)}\max_{r\in I}\left|e^{|g^\frac{1}{4}(r-r_0)|}\psi(r) \right| .
\end{split}
\end{equation}
The corresponding $\hat{H}$ now satisfies
\[|\hat{H}|\leq Cg(1-r)^2
\left|\psi\left(1-\rho (\ln g)g^{-\frac{1}{4}}\right)\right|+Cg^{\frac{1}{4}+4\rho}e^{4g^{\frac{1}{4}}(r_0-1)}\|(F,H)\|_0
+Cg^{\rho}e^{g^\frac{1}{4}(r_0-1)}\|(F,H)\|_2
\]
for $r\in \left(1-\rho (\ln g)g^{-\frac{1}{4}},1\right)$.
Hence, by Lemma \ref{lemBlowUP}, we infer that the corresponding auxiliary function $\hat{\psi}$ satisfies
\[\begin{split}\|\hat{\psi}\|_{L^\infty\left(1-\rho (\ln g)g^{-\frac{1}{4}},1\right)}\leq & C
\left|\psi\left(1-\rho (\ln g)g^{-\frac{1}{4}}\right)\right|
+Cg^{-\frac{1}{4}+4\rho}e^{4g^{\frac{1}{4}}(r_0-1)}\|(F,H)\|_0\\
&+Cg^{\rho-\frac{1}{2}}e^{g^\frac{1}{4}(r_0-1)}\|(F,H)\|_2.\end{split}
\]
So, via  (\ref{eqPvia}), and recalling the definition of $\hat{\psi}$, we find that
\begin{equation}\label{eqPremThanks}\begin{split}\|{\psi}\|_{L^\infty\left(1-\rho (\ln g)g^{-\frac{1}{4}},1\right)}\leq &
Cg^{2\rho-\frac{1}{4}}e^{2g^\frac{1}{4}(r_0-1)}\|(F,H)\|_0+Cg^{\rho-\frac{1}{2}}e^{g^\frac{1}{4}(r_0-1)}\|(F,H)\|_2\\&+Cg^{-2\rho}e^{g^\frac{1}{4}(r_0-1)}\max_{r\in I}\left|e^{g^\frac{1}{4}(r-r_0)}\psi(r) \right|.
\end{split}\end{equation}
Thus, by using this in (\ref{eqPremark}), we infer that
\[
\max_{r\in I}\left|e^{g^{\frac{1}{4}}(r-r_0)}\psi(r) \right| \leq Cg^{-\frac{1}{4}}\|(F,H)\|_0+Cg^{\rho-\frac{1}{2}}\|(F,H)\|_2
\]
Consequently, using (\ref{eqPremThanks}) once more to complete the estimate up to the boundary point $r=1$,   we arrive at the following estimate:
\emph{For any $\rho>0$, there exists a constant $C>0$ such that
\begin{equation}\label{eqPestimSecond}
\|(\phi,\psi) \|_1 \leq Cg^{-\frac{1}{4}}\|(F,H)\|_0+Cg^{\rho-\frac{1}{2}}\|(F,H)\|_2,
\end{equation}
provided that $g$ is sufficiently large.
}
\end{rmq}\section{The perturbation argument: Existence of a genuine solution}\label{secnonlinear}
 We seek a solution of  system (\ref{eq2comp}) as
\begin{equation}\label{eqGenForm}
(u,v)=(u_{ap},v_{ap})+(\varphi,\psi)
\end{equation}
with
\begin{equation}\label{eqDiri}
\varphi'(0)=\psi'(0)=\varphi(1)=\psi(1)=0.
\end{equation}

After rearranging terms, we find
that $(\varphi,\psi)$ has to satisfy
\begin{equation}\label{eqFixed}\left\{\begin{array}{c}
\mathcal{L}(\varphi,\psi)=-\mathcal{R}-N(\varphi,\psi),
 \\
   \\
\varphi'(0)=\psi'(0)=\varphi(1)=\psi(1)=0.\end{array}\right.
\end{equation}
where
$
\mathcal{L}$ is as in Proposition \ref{proLinear}, $\mathcal{R}$ is as in (\ref{eqRemainder}), and
\begin{equation}\label{eqN}
   N(\varphi,\psi)=
\left( \begin{array}{c}
  f(u_{ap}+\varphi)-f(u_{ap})-f'(u_{ap})\varphi+g u_{ap} \psi^2+ g \psi^2\varphi+2g v_{ap}\varphi\psi \\
    \\
h(v_{ap}+\psi)-h(v_{ap})-h'(v_{ap})\psi+g v_{ap} \varphi^2+ g \varphi^2\psi+2g u_{ap}\varphi\psi\\
\end{array}
\right).
\end{equation}

The main effort in this section will be placed in showing the following proposition.
\begin{prop}\label{proPerturb}
 Given $\gamma\in (0,1)$,
 there exists $M>1$ such that the problem (\ref{eqFixed}) admits a unique solution $(\varphi,\psi)\in \left[C(\overline{B_1})\cap C^2({B_1})\right]^2$ in the set
\[
\mathcal{B}_{M,g}=\left\{(\varphi,\psi)\in C[0,1]\times C[0,1]\ : \ \|(\varphi,\psi)\|_1\leq M |\ln g|^{5+\gamma} g^{-\frac{3}{4}} \right\},
\]
where the above norm is as in (\ref{eqPnorm0}), provided that $g$ is  sufficiently large.
\end{prop}
\begin{proof}
Let us define a mapping
\[T\ :\mathcal{B}_{M,g}\to\left[C(\overline{B_1})\cap C^2({B_1})\right]^2\]
by \[T(\varphi,\psi)\to(\bar{\varphi},\bar{\psi}),\]
where $(\bar{\varphi},\bar{\psi})$ is the unique solution of
\begin{equation}\label{eqFixed00}\left\{\begin{array}{c}
\mathcal{L}(\bar{\varphi},\bar{\psi})=-\mathcal{R}-N(\varphi,\psi),
 \\
   \\
\bar{\varphi}'(0)=\bar{\psi}'(0)=\bar{\varphi}(1)=\bar{\psi}(1)=0.\end{array}\right.
\end{equation} We point out that this mapping is well defined, for large $g$, thanks to Proposition \ref{proLinear}.

We claim that if $M>0$ is chosen sufficiently large, then $T$ maps $\mathcal{B}_{M,g}$ into itself, provided that $g$ is sufficiently large.
To this end, it is convenient to decompose $(\bar{\varphi},\bar{\psi})$ as the sum of a finite number of terms in the natural way.
The main term in this sum turns out to be $(\bar{\varphi}_0,\bar{\psi}_0)$ given by
\begin{equation}\label{eqFixed0}\left\{\begin{array}{c}
\mathcal{L}(\bar{\varphi}_0,\bar{\psi}_0)=-\mathcal{R},
 \\
   \\
\bar{\varphi}'_0(0)=\bar{\psi}'_0(0)=\bar{\varphi}_0(1)=\bar{\psi}_0(1)=0.\end{array}\right.
\end{equation}
By Propositions \ref{proRemainderGlobal} and \ref{proLinear}, we deduce that
\begin{equation}\label{eqN-1}
\|(\bar{\varphi}_0,\bar{\psi}_0)\|_1\leq Cg^{-\frac{1}{4}}\|\mathcal{R}\|_2\leq C|\ln g|^{5+\gamma}g^{-\frac{3}{4}}.
\end{equation}
The rest of the terms in the aforementioned sum come from neglecting $\mathcal{R}$ in (\ref{eqFixed}). Let us estimate some of the corresponding representative terms that come from the first line of $N$ in (\ref{eqN}).


For example let us consider $(\bar{\varphi}_1,\bar{\psi}_1)$ given by
\[
\mathcal{L}\left( \begin{array}{c}
   \bar{\varphi}_1 \\
    \bar{\psi}_1
\end{array}
\right)
=\left( \begin{array}{c}
  g \psi^2\varphi \\
    0
\end{array}
\right),
\]coupled with the boundary conditions (\ref{eqDiri}),
with $(\varphi,\psi)\in \mathcal{B}_{M,g}$. By Proposition \ref{proLinear},
we deduce that
\begin{equation}\label{eqN1}
\|(\bar{\varphi}_1,\bar{\psi}_1)\|_1\leq C g^\frac{3}{4} \|( \psi^2\varphi,0)\|_2.
\end{equation}
Then, to estimate the righthand side, we take advantage of the cancelation properties. For $r\geq r_0$, we find that
\[
\left(1+|g^{\frac{1}{4}}(r-r_0)|^{1+\gamma} \right)\psi^2|\varphi|\leq \left(1+|g^{\frac{1}{4}}(r-r_0)|^{1+\gamma} \right)e^{-2g^\frac{1}{4}(r-r_0)}\|(\varphi,\psi)\|_1^2|\varphi|
\leq C\|(\varphi,\psi)\|_1^3.
\]
On the other side, for $r\leq r_0$ we have
\[
e^{g^\frac{1}{4}|r-r_0|}\psi^2|\varphi|\leq \|(\varphi,\psi)\|_1 \psi^2 \leq C\|(\varphi,\psi)\|_1^3.
\]By the above two relations, we obtain that
\[
\|(\varphi\psi^2,0)\|_2\leq C\|(\varphi,\psi)\|_1^3.
\]
Hence, by (\ref{eqN1}) and the above relation, recalling that $(\varphi,\psi)\in \mathcal{B}_{M,g}$, we infer that
\begin{equation}\label{eqN0}\|(\bar{\varphi}_1,\bar{\psi}_1)\|_1\leq CM^3 |\ln g|^{15+3\gamma}g^{-\frac{3}{2}},
\end{equation}
with the constant $C$ being independent of both $M$ and $g$.
We can estimate analogously the other coupled terms from (\ref{eqN}), using that
\[|u_{ap}|\leq C g^{-\frac{1}{4}}e^{2g^\frac{1}{4}(r-r_0)}\ \textrm{if}\ r\leq r_0,\
|u_{ap}|\leq C (g^{-\frac{1}{4}}+r-r_0)\ \textrm{if}\ r\geq r_0
\]
and the analogous estimates for $v_{ap}$ (these follow directly from their construction). For instance, for $r\geq r_0$, we find that
\[
    \begin{array}{rcl}
      \left(1+|g^{\frac{1}{4}}(r-r_0)|^{1+\gamma} \right)u_{ap}\psi^2 & \leq & C g^{-\frac{1}{4}}  \left(1+|g^{\frac{1}{4}}(r-r_0)|^{1+\gamma} \right)^2e^{-2g^\frac{1}{4}(r-r_0)}\|(\varphi,\psi)\|_1^2 \\
        &   &   \\
        & \leq  & Cg^{-\frac{1}{4}}\|(\varphi,\psi)\|_1^2. \\
    \end{array}
\]
On the other side, for $r\leq r_0$ we have
\[
e^{g^\frac{1}{4}|r-r_0|}u_{ap}\psi^2\leq  Cg^{-\frac{1}{4}}\psi^2 \leq Cg^{-\frac{1}{4}}\|(\varphi,\psi)\|_1^2.\] So, we get
\begin{equation}\label{eqN2}
\|(gu_{ap}\psi^2,0)\|_2\leq  Cg^{\frac{3}{4}}\|(\varphi,\psi)\|_1^2\leq CM^2|\ln g|^{10+2\gamma}g^{-\frac{3}{4}},\end{equation}with $C$ independent of both $M$ and $g$.
Estimating the term $(\bar{\varphi}_2,\bar{\psi}_2)$ that is defined by
\[
\mathcal{L}\left( \begin{array}{c}
   \bar{\varphi}_2 \\
    \bar{\psi}_2
\end{array}
\right)
=\left( \begin{array}{c}
 2 gv_{ap} \psi \varphi \\
    0
\end{array}
\right),
\]coupled with the boundary conditions (\ref{eqDiri}),
with $(\varphi,\psi)\in \mathcal{B}_{M,g}$, is a bit tricky.
This is because, for $r\leq r_0$, we can only show that
\[
ge^{g^\frac{1}{4}|r-r_0|}v_{ap}|\varphi \psi|\leq Cg(g^{-\frac{1}{4}}+|r-r_0|)\|(\varphi,\psi)\|_1^2.
\]
The above estimate and the bound
\[
g\left(1+|g^\frac{1}{4}(r-r_0)|^{1+\gamma} \right)v_{ap}|\varphi \psi|\leq Cg^\frac{3}{4}\left(1+|g^\frac{1}{4}(r-r_0)|^{1+\gamma} \right)
e^{-g^\frac{1}{4}|r-r_0|}\|(\varphi,\psi)\|_1^2\leq Cg^\frac{3}{4}\|(\varphi,\psi)\|_1^2,
\]
which holds for $r\geq r_0$, yield that
\[
\|(gv_{ap}\varphi\psi,0)\|_2\leq g\|(\varphi,\psi)\|_1^2,
\]
which is worse than (\ref{eqN2}) and, as it turns out, not sufficient for our purposes.
Nevertheless, arguing as above, we observe that
\[
\|(gv_{ap}\varphi\psi ,0)\|_0\leq  Cg^{\frac{3}{4}}\|(\varphi,\psi)\|_1^2\leq CM^2|\ln g|^{10+2\gamma}g^{-\frac{3}{4}},\]
with $C$ independent of both $M$ and $g$, where the norm in the lefthand side was defined in the beginning of Remark \ref{remP}. Hence, by the a-priori estimate (\ref{eqPestimSecond}), for any $\rho>0$, we infer that
\begin{equation}\label{eqN3}
\|(\bar{\varphi}_2,\bar{\psi}_2)\|_1\leq C g^\frac{1}{2}\|(\varphi,\psi)\|_1^2+Cg^{\rho+\frac{1}{2}}\|(\varphi,\psi)\|_1^2\leq CM^2|\ln g|^{10+2\gamma}g^{\rho-1},
\end{equation} for some constant $C$ that is independent of both $M$ and $g$.

Concerning the uncoupled terms in the first line of $N$ in (\ref{eqN}), we first note that the following estimate
\[
| f(u_{ap}+\varphi)-f(u_{ap})-f'(u_{ap})\varphi |\leq C \varphi^2
\]
holds. Then, for $r\geq r_0$, we find that
\[
      \left(1+|g^{\frac{1}{4}}(r-r_0)|^{1+\gamma} \right)\varphi^2  \leq  C g^{\frac{1+\gamma}{4}}\varphi^2,\]
while for $r\leq r_0$ we have
\[
e^{g^\frac{1}{4}|r-r_0|} \varphi^2\leq  C\|(\varphi,\psi)\|_1|\varphi| .\]
Thus, we get that
\begin{equation}\label{eqNf}
  \|(f(u_{ap}+\varphi)-f(u_{ap})-f'(u_{ap})\varphi,0)\|_2\leq Cg^{\frac{1+\gamma}{4}}\|(\varphi,\psi)\|_1^2\leq CM^2|\ln g|^{10+2\gamma}g^{\frac{1+\gamma}{4}}g^{-\frac{3}{2}},
\end{equation}where the constant $C$ is independent of both $M$ and $g$.

By virtue of (\ref{eqN-1}),  (\ref{eqN0}), (\ref{eqN3}), and applying Proposition \ref{proLinear} with (\ref{eqN2}),   (\ref{eqNf}) in hand, as well as the corresponding
estimates from the second line of (\ref{eqN}), we infer from (\ref{eqFixed00}) that
\[
\|(\bar{\varphi},\bar{\psi})\|_1\leq C|\ln g|^{5+\gamma}g^{-\frac{3}{4}}+CM^3 |\ln g|^{15+3\gamma}g^{-\frac{3}{2}}+CM^2|\ln g|^{10+2\gamma}g^{\rho-1}+CM^2|\ln g|^{10+2\gamma}g^{\frac{1+\gamma}{4}}g^{-\frac{7}{4}},
\]
holds for any $\rho>0$ and some constant $C$ that is independent of both $M$ and $g$. Thus, choosing first a sufficiently small $\rho>0$, we can then choose a large $M>0$ such that
\[
\|T(\varphi,\psi)\|_1=\|(\bar{\varphi},\bar{\psi})\|_1\leq M|\ln g|^{5+\gamma}g^{-\frac{3}{4}},
\]
provided that $g$ is sufficiently large. In other words, $T$ maps $\mathcal{B}_{M,g}$ into itself.

Working as above, we can also show that $T:\mathcal{B}_{M,g}\to \mathcal{B}_{M,g}$ is a contraction. Hence, by the contraction mapping theorem, we conclude that $T$ has a unique fixed point in $\mathcal{B}_{M,g}$ which provides the desired solution to  (\ref{eqFixed}).

\end{proof}

\subsection{Proof of Theorem \ref{thm}}
Armed with the above information, we are now ready for the proof of Theorem \ref{thm}.
\begin{proof}
  It follows from Proposition \ref{proPerturb} that (\ref{eq2comp}) has a solution $(u,v)$ such that
  \begin{equation}\label{eqApFinale}
  u=u_{ap}+O\left(|\ln g|^{6}g^{-\frac{3}{4}} \right)\ \textrm{and}\ v=v_{ap}+O\left(|\ln g|^{6}g^{-\frac{3}{4}} \right),
  \end{equation}
  uniformly in $\Omega=B_1$, as $g\to +\infty$. The estimates in the theorem follow directly from the main order terms in the construction of the approximate solution $(u_{ap},v_{ap})$.

  It remains to show that $u$ and $v$ are positive. This task will take the rest of the proof. Since $u_0>0$ in $(r_0,1)$ and vanishes in a linear fashion at the endpoints (by Hopf's
  boundary point lemma), we obtain from the global estimate of Theorem \ref{thm}, i.e.,
  \begin{equation}\label{eqAnalog}
  u=u_0+O(g^{-\frac{1}{4}}), \ \textrm{uniformly in}\ (r_0,1),
  \end{equation}
  that there exists an $L>1$ such that
\begin{equation}\label{eqPos1}
u>0\ \textrm{for}\ r\in(r_0+Lg^{-\frac{1}{4}},1-Lg^{-\frac{1}{4}}),
\end{equation}
provided that $g$ is sufficiently large.
From the proof of Proposition \ref{proPerturb}, recalling the construction of the outer approximate solution, we can write
\begin{equation}\label{eqPos1+}
u=u_{ap}+\varphi\geq c(1-r)+\varphi,\ \ r\in \left(\frac{r_0+1}{2},1 \right).
\end{equation}
Since $\varphi$ is uniformly small, the mean value theorem provides an $r_g$ in the above interval such that
$\varphi'(r_g)=o(1)$. In turn, integrating the uniform estimate
 $(r^{N-1}\varphi')'=o(1)$ (recall (\ref{eqFixed})) from $r_g$ to $r$ we get that
\[
\varphi'=o(1) \ \textrm{uniformly in}\ \left(\frac{r_0+1}{2},1 \right).
\]
Then, integrating the above relation from $r$ to $1$ (recall that $\varphi(1)=0$), we infer from (\ref{eqPos1+}) that
\[
u\geq c(1-r),\ \ r\in\left(\frac{r_0+1}{2},1 \right).
\]
So, the relation (\ref{eqPos1}) extends up to $r=1$. On the other side, using the inner estimate of Theorem \ref{thm}, we can extend the aforementioned relation from
$r_0-g^{-\frac{1}{4}}$ to $1$. To complete it in the whole ball $B_1$ we will employ the maximum principle.
 We note that $u$ satisfies a linear equation of the form
\begin{equation}\label{eqanalogII}
-\Delta u+\left(gv^2+p(r)\right)u=0,\ \ |x|<r_0-g^{-\frac{1}{4}},
\end{equation}
where
\[
p(r)=\left\{\begin{array}{ll}
              \frac{f(u)}{u} & \textrm{if}\ u(r)\neq 0, \\
              0 & \textrm{elsewhere}.
            \end{array}
\right.
\]
We point out that, since $f(0)=0$, we have that
\[
|p|\leq C,\ \ |x|<r_0-g^{-\frac{1}{4}}.
\]
Now, since \begin{equation}\label{eqcorresponding}gv^2\geq c g^{\frac{1}{2}}\end{equation} therein (from the analog of (\ref{eqAnalog})) and $u>0$ on $r=r_0-g^{-\frac{1}{4}}$, we deduce by that maximum principle that
\[
u>0\ \ \textrm{in}\ |x|<r_0-g^{-\frac{1}{4}},
\]
as desired.

We can show that $v>0$ in $B_1$ analogously. The main difference is in the vicinity of the boundary $r=1$, where $u$ vanishes and the corresponding relation to (\ref{eqcorresponding}) fails. Nevertheless, thanks to (\ref{eqAnalog}) and the comments leading to it,   there exists a constant $K>1$ such that
\begin{equation}\label{eqPotential}
gu^2\geq g^\frac{1}{2},\ \ r\in (r_0+Kg^{-\frac{1}{4}}, 1-Kg^{-\frac{1}{4}}),
\end{equation}
provided that $g$ is sufficiently large. Near $r=1$ we can apply an argument that is inspired by the maximum principle for domains with small volume.
By  the first assertion of Lemma \ref{lemBlowUP} (which also applies for mixed boundary conditions), after a regular perturbation argument using that
\[u^2=u_{ap}^2+O\left(|\ln g|^{6}g^{-\frac{3}{4}} \right)\ \ (\textrm{recall\ (\ref{eqApFinale})}),
\] we find that the principal eigenvalue $\lambda_1$ of the following eigenvalue problem
\[\left\{\begin{array}{l}
    -{\psi}''-\frac{N-1}{r}{\psi}'+gu^2{\psi}=\lambda \psi,\ \ r\in \left(1- (\ln g)g^{-\frac{1}{4}},1\right), \\
    \psi'\left(1- (\ln g)g^{-\frac{1}{4}}\right)=\psi(1)=0,
  \end{array}\right.
\]
satisfies
\begin{equation}\label{eqlambda}
\lambda_1\geq c g^{\frac{1}{2}}.
\end{equation}
So, for any $\psi \in W_0^{1,2}(I\cup J)$ with \[I=\left(r_0+Kg^{-\frac{1}{4}}, 1-(\ln g)g^{-\frac{1}{4}}\right),\ \ J=\left(1-(\ln g)g^{-\frac{1}{4}},1 \right),\] we  find from (\ref{eqPotential}) and (\ref{eqlambda}) that
\begin{equation}\label{eqsmall}
\int_{I\cup J}^{}\left\{|\nabla\psi|^2+gu^2\psi^2\right\}r^{N-1}  dr\geq cg^\frac{1}{2}\int_{I}^{}\psi^2r^{N-1}  dr+\lambda_1\int_{J}^{}\psi^2r^{N-1}  dr
\geq cg^\frac{1}{2}\int_{I\cup J}^{}\psi^2r^{N-1}  dr.
\end{equation}
In analogy to (\ref{eqanalogII}), we see that $v$ satisfies a linear equation of the form
\begin{equation}\label{eqanalogIII}
-\Delta v+\left(gu^2+q(r)\right)v=0,\ \ r\in\left(r_0+Kg^{-\frac{1}{4}},1\right)\ \textrm{with}\ |q|\leq C.
\end{equation}
Moreover, we have
\[  v(r_0+Kg^{-\frac{1}{4}})>0 \ \textrm{(by the inner estimate of Theorem \ref{thm})}\ \textrm{and} \ v(1)=0.\]
Testing (\ref{eqanalogIII}) by $v^-\in W^{1,2}_0(I\cup J)$  yields
\[
\int_{I\cup J}^{}\left\{|\nabla v^-|^2+gu^2(v^-)^2\right\}r^{N-1}dr=-\int_{I\cup J}^{}q(r)(v^-)^2r^{N-1}dr\leq C \int_{I\cup J}^{}(v^-)^2r^{N-1}dr.
\]
Then, by using (\ref{eqsmall}) with $\psi=v^-\in W^{1,2}_0(I\cup J)$, we conclude that $v^-\equiv 0$ if $g$ is sufficiently large. Finally, the latter relation and the strong maximum principle in (\ref{eqanalogIII}) yield the desired positivity of $v$.
\end{proof}
\section{Applications of the main result}\label{secapplications}

 Let us now give briefly some applications of Theorem \ref{thm}. As it was already pointed out earlier,  in the case $f\equiv h$ and $f$ is odd the limit problem becomes (\ref{eqlimRed}). It is known that when $f(u)=\lambda u-u^{2p+1}$, $\lambda \geq 0$ and
  $p$ is such that \begin{equation}\label{eqExpon}1<2p+1<\frac{N+2}{N-2}\ \textrm{if}\ N\geq 3,\ p>0\ \textrm{if}\ N=2,\end{equation} then   a radial solution $w$ to (\ref{eqlimRed})  is unique and non-degenerate in the radial class provided that
  \begin{itemize}
 \item $w$ is positive, $\lambda \neq 0$ and $\Omega$ is an annulus or the exterior of a ball, see \cite{FelMarTan};
 \item $w$ is positive, $\lambda=0$ and $\Omega$ is a ball or an annulus,  see \cite{Pac};
 \item $w$ is positive, $\lambda \neq 0$ and $\Omega$ is a ball, see \cite{AftPac};
 \item $w$ is a nodal solution with two nodal regions, $\lambda=0$, see \cite{SanPac}.
 \end{itemize}

 We also refer to \cite{harribi,ShioWat} for more general results concerning the function $f$. We point out that such solutions can be shown to exist by variational methods.

  Thanks to these previous results,
 we see that our result applies in the case $f(u)=-u^{2p+1}$ with $p$ as in (\ref{eqExpon}), and $\Omega$ a ball or an annulus. In a related topic, let us point out that when $\Omega$ is the whole $N$-dimensional space, $N\geq 3$, and $f(u)=u-|u|^{p-1}u$ with $1<p<\frac{N+2}{N-2}$ sufficiently close to $\frac{N+2}{N-2}$, Ao, Wei and Yao \cite{AoWeYa} constructed radial solutions with $k\geq 1$ nodes to (\ref{eqlimRed}) that tend to zero as $r\to \infty$. Moreover, they established that their solutions are unique and non-degenerate. Our theorem, with only minor modifications in the proof, can produce a corresponding solution to (\ref{eq2comp}) for large $g$, starting from such a one-node solution.

\medskip

\textbf{Acknowledgements.} The first author is supported by MIS F.4508.14(FNRS) and PDR T.1110.14F(FNRS). This project has received funding from the Hellenic Foundation for
Research and Innovation (HFRI) and the General Secretariat for
Research and Technology (GSRT), under grant agreement No 1889.
The second author would like to thank the mathematics department of
ULB for its hospitality where part of this work was carried out.

\end{document}